\documentclass[letterpaper, english]{amsart}

\title{The de Rham functor for logarithmic D-modules}
\author{Clemens Koppensteiner}

\usepackage[utf8]{inputenc}
\usepackage[T1]{fontenc}
\usepackage{lmodern}
\usepackage{microtype}

\usepackage{amsmath,amsfonts,amssymb}
\usepackage{mathtools,extpfeil,xfrac}
\usepackage{amsthm,thmtools}

\usepackage{tikz-cd}

\usepackage{enumitem}
\setlist[enumerate,1]{label=(\roman*)}
\setlist[enumerate,2]{label=(\alph*)}
\setlist[enumerate,3]{label=(\Alph*)}
\setlist[enumerate,4]{label=\arabic*.}

\usepackage[
    backend=biber,
    style=alphabetic,
    maxnames=100,
    maxalphanames=100,
]{biblatex}

\DeclareLabelalphaTemplate{
    \labelelement{
        \field[final]{shorthand}
        \field{label}
        \field[varwidthlist,strside=left,names=4]{labelname}
    }
}

\DeclareFieldFormat{labelalpha}{#1}
\DeclareFieldFormat{extraalpha}{#1}

\usepackage{mparhack}  
\usepackage{todonotes}
\usepackage{csquotes}
\usepackage[overload]{textcase}
\usepackage{url}

\usepackage[colorlinks=false,unicode=true,bookmarksdepth=2]{hyperref}
\usepackage{bookmark}

\theoremstyle{plain}
\newtheorem{Theorem}{Theorem}[section]
\newtheorem*{Theorem*}{Theorem}

\newtheorem{Proposition}[Theorem]{Proposition}
\newtheorem{Corollary}[Theorem]{Corollary}
\newtheorem{Lemma}[Theorem]{Lemma}

\theoremstyle{definition}
\newtheorem{Definition}[Theorem]{Definition}

\theoremstyle{remark}
\newtheorem{Remark}[Theorem]{Remark}

\declaretheorem[name=Example,sibling=Theorem,qed={\lower-0.3ex\hbox{$\bigcirc$}}]{Example}

\newcommand\NN{\mathbb{N}}
\newcommand\ZZ{\mathbb{Z}}

\newcommand\CC{\mathbb{C}}

\newcommand\cat\mathbf
\newcommand\Hom{\operatorname{Hom}}

\newcommand\Ext{\operatorname{Ext}}
\newcommand\id{\operatorname{id}}   
\newcommand\dual\vee
\newcommand\catMod[2][]{\cat{Mod}_{#1}(#2)}     
\newcommand\opcat{{\mathrm{op}}}

\newcommand\RR{\mathbb{R}}

\newcommand\catD[2][]{\cat{D}^{#1}(#2)}
\newcommand\catDb[2][]{\cat{D}^{\mathrm{b}}_{#1}(#2)}
\newcommand\catDp[1]{\catD[+]{#1}}

\newcommand\supp{\operatorname{supp}}   
\newcommand\sheaf\mathcal
\newcommand\sheafHom{\underline{\Hom}}  
\newcommand\sO{\sheaf{O}}           
\newcommand\sT{\sheaf{T}}           
\newcommand\sD{\sheaf{D}}           
\newcommand\canon\omega             
\newcommand\rigidDC{\sheaf R}       
\newcommand\twist\mu                
\newcommand\DVerdier{\mathbb{D}}    
\newcommand\DSerre{\mathbb{D}^{\mathrm{Se}}}    
\newcommand\spf\bullet              
\newcommand\Sp{\operatorname{Sp}}   
\newcommand\Odc{\omega'}            
\newcommand\Ch{\operatorname{Ch}}   
\newcommand\Omod{\sO\mbox{-mod}}    
\newcommand\DR{\operatorname{DR}}   
\newcommand\tDR{\widetilde{\operatorname{DR}}}  

\newcommand\logM{M}                             
\newcommand\logMbar{\overline{\logM}}           
\newcommand\logMbargp{\logMbar^{\mathrm{gp}}}   
\newcommand\logK{K}                             

\newcommand\Clog{\CC^{\log}}        
\newcommand\Olog{\sO^{\mathrm{log}}}
\newcommand\tOlog{\tilde{\sO}^{\mathrm{log}}}
\newcommand\tDlog{\tilde{\sD}^{\mathrm{log}}}
\newcommand\tOmegalogi[1]{\tilde{\Omega}^{#1,\mathrm{log}}}
\newcommand\tDVerdierlog{\tilde{\mathbb{D}}^{\mathrm{log}}}
\newcommand\catSh[1]{\cat{Sh}(#1)}

\newcommand\catLogLocSys[1]{\cat{L}^\Lambda_{\mathrm{coh}}(\Clog_{#1})}
\newcommand\catClogMod[1]{\catMod{\Clog_{#1}}}      
\newcommand\catClogModTau[1]{\catMod[\tau]{\Clog_{#1}}}         
\newcommand\catClogModMon[1]{\catMod[\Lambda]{\Clog_{#1}}}      
\newcommand\catDbClogMod[1]{\catDb{\Clog_{#1}}}     
\newcommand\catDbfClogMod[1]{\cat{D}^{\mathrm{b}}_f(\Clog_{#1})}    

\newcommand\catDbClogModMon[1]{\catDb[\Lambda]{\Clog_{#1}}}
\newcommand\logdc{\omega}                           
\newcommand\tlogdc{\widetilde{\omega}}              
\newcommand\tDD{\widetilde{\mathbb{D}}}             

\newcommand\catCohOMod[1]{\cat{Coh}(\sO_{#1})}   
\newcommand\catDbOMod[1]{\cat{D}^{\mathrm{b}}(\sO_{#1})}    

\newcommand\catDMod[1]{\catMod{\sD_{#1}}}           
\newcommand\catDModOp[1]{\catMod{\sD^\opring_{#1}}} 

\newcommand\catHolDMod[1]{\cat{Hol}(\sD_{#1})}   
\newcommand\catCohDMod[1]{\cat{Coh}(\sD_{#1})}   
\newcommand\catQCDMod[1]{\cat{QCoh}(\sD_{#1})}   

\newcommand\catDbDMod[1]{\cat{D}^{\mathrm{b}}(\sD_{#1})}    
\newcommand\catDpDMod[1]{\cat{D}^{+}(\sD_{#1})}             

\newcommand\catDbqcDMod[1]{\cat{D}^{\mathrm{b}}_{\mathrm{qc}}(\sD_{#1})}    
\newcommand\catDmqcDMod[1]{\cat{D}^{\mathrm{-}}_{\mathrm{qc}}(\sD_{#1})}    

\newcommand\catDbcohDMod[1]{\cat{D}^{\mathrm{b}}_{\mathrm{coh}}(\sD_{#1})}  
\newcommand\catDbcohDModop[1]{\cat{D}^{\mathrm{b}}_{\mathrm{coh}}(\sD_{#1}^\opring)}    

\newcommand\catDbholDMod[1]{\cat{D}^{\mathrm{b}}_{\mathrm{h}}(\sD_{#1})}    

\newcommand\pt{\mathrm{pt}}                 
\newcommand\as[1]{\mathbb{A}^{#1}}          
\newcommand\ps[1]{\mathbb{P}^{#1}}          
\newcommand\ulps[1]{\ul{\mathbb{P}}^{#1}}   
\newcommand\logdim{\operatorname{logdim}}   
\newcommand\KN[1]{#1_{\mathrm{log}}}        
\newcommand\an{{\mathrm{an}}}               
\newcommand\Spec{\operatorname{Spec}}       

\newcommand\ul\underline
\newcommand\opring{{\mathrm{op}}}   
\newcommand\gr{\operatorname{gr}}   
\newcommand\gp{{\mathrm{gp}}}
\newcommand\from\leftarrow
\newcommand\isoto{\xrightarrow{\sim}}
\newcommand\res[2]{\mathchoice{\left.#1\right|_{#2}}{#1|_{#2}}{#1|_{#2}}{#1|_{#2}}} 
\newcommand\rquot[2]{
    \mathchoice%
        {\left.#1\kern-0.2ex\middle/\kern-0.3ex\lower0.7ex\hbox{$\displaystyle #2$}\right.}%
        {\left.#1\middle/#2\right.}%
        {\left.#1\middle/#2\right.}%
        {\left.#1\middle/#2\right.}%
}

\begin{document}

\begin{abstract}
    In the first part we deepen the six-functor theory of (holonomic) logarithmic D-modules, in particular with respect to duality and pushforward along projective morphisms.
    Then, inspired by work of Ogus, we define a logarithmic analogue of the de Rham functor, sending logarithmic D-modules to certain graded sheaves on the so-called Kato--Nakayama space.
    For holonomic modules we show that the associated sheaves have finitely generated stalks and that the de Rham functor intertwines duality for D-modules with a version of Poincar\'e--Verdier duality on the Kato--Nakayama space.
    Finally, we explain how the grading on the Kato--Nakayama space is related to the classical Kashiwara--Malgrange V-filtration for holonomic D-modules.
\end{abstract}

\maketitle

\setcounter{tocdepth}{1}
\tableofcontents

\section{Introduction}\label{sec:intro}

It is a well known observation that in many regards compact spaces behave better that non-compact ones.
Thus in algebraic geometry, using Hironaka's Theorem, one often finds oneself having to consider a smooth (compact) variety $\ul X$ together with a normal crossings divisor $D$.
Objects living on $\ul X$ then have to take this boundary into account.
A classical example of this are meromorphic connections which play an integral role in Deligne's Riemann--Hilbert correspondence.

In logarithmic geometry (regular) meromorphic connections get recast as integrable logarithmic connections on the smooth log variety $X = (\ul X, D)$.
That is, one constructs the logarithmic cotangent bundle $\Omega_X^1$ and defines a log connection to be a coherent sheaf $\sheaf M$ on $X$ together with a $\CC$-linear map $\nabla \colon \sheaf M \to \sheaf M \otimes_{\sO_X} \Omega_X^1$ satisfying the Leibniz rule.
One wishes to classify such objects via a logarithmic analogue of the Riemann--Hilbert correspondence.

To see why this takes some effort let us consider some examples of integrable connections on $\as 1$ with the log structure given by $D = \{ x = 0\}$.
For any complex number $\lambda$ one can define a connection $\nabla_\lambda$ on $\sO_X$ by $\nabla_\lambda(1) = \lambda \otimes \frac{dx}{x}$.
One notes that unlike the similarly defined connections on $\as 1 \setminus \{0\}$ (where $\nabla_\lambda \cong \nabla_\lambda'$ if $\lambda = \lambda' \bmod \ZZ$), in the logarithmic setting one has $\nabla_\lambda \ncong \nabla_{\lambda'}$ for any $\lambda \ne \lambda'$.

The first and most obvious problem when looking for a generalization of the Riemann--Hilbert correspondence to this setting is the fact that $\as 1$ does not support any non-trivial locally constant sheaves.
Thus Kato and Nakayama \cite{KatoNakayama} define an auxiliary topological space $\KN X$, living over $X$, on which the correspondence takes place.
This space, now usually called the Kato--Nakayama space of $X$, is essentially the real blowup of $X$ along $D$.
Thus in the example at hand one replaces the origin of $\as 1$ by a circle.

The space $\KN X$ thus has a non-trivial fundamental group and hence interesting locally constant sheaves.
However monodromy around the circle can only record $\lambda \mod \ZZ$.
Thus Ogus further enhances $\KN X$ with a sheaf of gradings $\Lambda$ \cite{Ogus}.
In our example this sheaf is locally constant with stalk $\CC$ on the added circle over the origin and $0$ elsewhere.
Then the sheaf corresponding to the connection $\nabla_\lambda$ on $\sO_X$ would be the local system with monodromy $e^{2\pi i \lambda}$ placed in degree $\lambda$.

However even this is not quite sufficient, as logarithmic connections are not necessarily vector bundles.
For example the same formula as above defines connections on $\sO_X / (x^2)$.
In order to record the possibly interesting action of $x$ at the origin Ogus further enhances $\KN X$ with a $\Lambda$-graded sheaf of rings $\Clog_X$.
For $\as 1$ this sheaf has stalk $\CC[t]$ with $\deg t = -1$ on the circle over the origin and is the constant sheaf $\CC$ elsewhere (the restriction map is given by $t \mapsto 1$).
Thus the target of the Riemann--Hilbert correspondence should be $\Lambda$-graded $\Clog_X$-modules on $\KN X$, with some kind of \enquote{coherency} condition \cite[Definition~3.2.4]{Ogus}.
For example the connections $\nabla_\lambda$ on $\sO_X$ correspond to $\Clog_X \otimes_\CC \sheaf F$, where $\sheaf F$ is the local system with monodromy $e^{2\pi i \lambda}$ placed in degree $\lambda$ as before.

In classical theory, the Riemann--Hilbert correspondence has long been generalized to (regular holonomic) D-modules.
It is thus an obvious question whether the above theory can be generalized to logarithmic D-modules.

The study of holonomic log D-modules was started in \cite{KT} where in particular one examined the interaction of holonomicity and duality.
To have any hope of obtaining a generalization of the Riemann--Hilbert correspondence, one needs a good understanding of the six functor formalism for both sides of the correspondence.
For graded sheaves on the Kato--Nakayama space, this was done via the new theory of \enquote{graded topological spaces} explored in \cite{Koppensteiner:graded_top}, where one in particular generalized Poincar\'e--Verdier duality to this setting.
For logarithmic D-modules this is done in the first part of the present paper.

\subsection{Functors on D-modules}

Let $f\colon X \to Y$ be a morphism of smooth log varieties.
The definition of the pushforward $f_\spf$ was already given in \cite{KT}, and is virtually identical to the corresponding definition for ordinary D-modules.
We study some properties of this pushforward, and in particular prove in Section~\ref{sec:pushforward_hol} that if $f$ is projective then $f_\spf$ preserves holonomicity.

The definition of the pullback $f^!$ is more subtle.
It is constricted by two factors: If $f$ is proper, then $f^!$ should be right adjoint to $f_{\spf}$, and in any case it should match the $!$-pullback for graded topological spaces.
In particular for $p\colon X \to \pt$ the structure map, $p^!\sO_{\pt}$ should correspond via the logarithmic Riemann--Hilbert correspondence to a certain subsheaf $\tlogdc_{\KN X}$ of the dualizing complex of the Kato--Nakayama space viewed as a graded topological spaces.
We compute this dualizing complex in Section~\ref{sec:Clog}, where it turns out that it is a complex of the type of coherent $\Clog$-modules considered by Ogus.
Thus we know that the Riemann--Hilbert correspondence of Ogus must match any proposed definition of $p^!\sO_{\pt}$ with $\tlogdc_{\KN X}$.
It follows that $p^!$ is not simply a shift of the naive pullback $f^*$ as is the case in the classical theory.

The definition of $f^!$ is given in Section~\ref{sec:pullback}.
It is inspired by the computation of the dualizing complex done in Section~\ref{sec:duality}.

In \cite{KT} one only uses the existence and properties of a rigid dualizing complex for $\sD_X$.
For further results it is useful to have an explicit formula.
Generalizing work of Chemla \cite{Chemla:2004:RigidDualizingComplexForQuantumEnvelopingAlgebras}, we show that the rigid dualizing complex is given by $\sD_X \otimes_{\sO_X} \twist_X$, where $\twist_X$ is the left $\sD_X$-module corresponding to the Grothendieck dualizing complex of $\ul X$ (which is naturally a right $\sD_X$-module).
One notes that in general the underlying classical variety $\ul X$ may have singularities, so that $\twist_X$ may be nontrivial.
We also show that duality commutes with pushforward along projective morphisms (Theorem~\ref{thm:proper_duality}).

\subsection{The logarithmic de Rham functor}

In order to construct the Riemann--Hilbert correspondence for logarithmic connections, Ogus introduces an additional $\Lambda$-graded sheaf of rings on $\KN X$, the enhanced structure sheaf $\tOlog_X$ \cite[Section~3.3]{Ogus}.
If $\tau\colon \KN X \to X$ is the projection and $\sheaf F$ is any $\sO_X$-module, one sets $\tau^*\sheaf F = \tOlog_X \otimes_{\tau^{-1}\sO_X} \tau^{-1}\sheaf F$.
In Section~\ref{sec:Olog} we define sheaf of differential operators $\tDlog_X$ on $\KN X$ so for any $\sD_X$-module $\sheaf M$ the pullback $\tau^*\sheaf M$ is naturally a $\tDlog_X$-module.

In analogy with the classical definition, the logarithmic de Rham functor is given by
\[
    \tDR_X(\sheaf M) = \tau^*\canon_X \otimes_{\tDlog_X} \tau^*\sheaf M.
\]
If $\sheaf M$ is $\sO_X$-coherent this agrees up to a shift with \cite[Theorem~3.4.2]{Ogus}.
We prove the following important finiteness theorem for $\tDR$.

\begin{Theorem*}
    Let $\sheaf M \in \catDbholDMod{X}$ be holonomic.
    Then the stalks of $\tDR_X(\sheaf M)$ are finitely generated over the corresponding stalk of $\Clog_X$.
\end{Theorem*}

Further, as a check that this new de Rham functor respects the six functor formalism we show the following fundamental result.

\begin{Theorem*}
    The functor $\tDR_X$ intertwines (rigid) duality for logarithmic D-modules with Poincar\'e--Verdier duality for the graded topological space $\KN X$.
\end{Theorem*}

In the process of proving these theorems we also show that $\tDR_X$ commutes with pushforward along projective morphisms whenever possible.
(We note that such pushforward can only preserve holonomicity when the log structure on the target is \enquote{no worse then the log structure on the source,} see Theorem~\ref{thm:projective_pf_preserves_holonomic}.)

In the final section we consider a smooth variety $\ul X$ together with a smooth divisor $D = \{ t =0\}$.
If $\sheaf M$ is a holonomic $\sD_{\ul X}$-module, then one can endow it with the Kashiwara--Malgrange V-filtration given by the eigenvalues of the action of $t\frac{\partial}{\partial t}$.
On the other hand $\tDR_X(\sheaf M)$ can be filtered as a $\Clog_X$-module by the $\Lambda$-grading.
We show that $\tDR_X$ exactly matches the V-filtration with the $\Lambda$-grading filtration.
This gives a nice interpretation for the appearance of the grading sheaf $\Lambda$.
In future work we are interested in extending the classical V-filtration to a multi-filtration for holonomic $\sD_X$-modules on arbitrary smooth log varieties.
This filtration would then conjecturally match the filtration by $\Lambda$-grading on the de Rham functor.

\subsection{Acknowledgements}

The author would like to thank David Ben-Zvi, Pierre Deligne and Dmitry Vaintrob for enlightening conversations regarding this project and in particular Mattia Talpo for answering many questions about logarithmic geometry.
He would also like to thank an anonymous referee for providing many valuable comments and suggestions.
The author was supported by the National Science Foundation under Grant No.~DMS-1638352.

\subsection{Conventions}

We will need to consider both complex algebraic and analytic varieties.
In order to present a unified perspective whenever possible, the word \enquote{variety} will always encapsulate both settings unless otherwise specified.
We will be mostly concerned with derived categories of D-modules and other sheaves.
Hence, unless noted otherwise, all functors will be derived and we will not use the signifiers $\mathbb{L}$ and $\mathbb{R}$.

\section{Logarithmic D-modules}

In this paper we are concerned with D-modules on a smooth, potentially idealized, complex log variety $X$.
For algebraic $X$ these were introduced in \cite{KT}.
In this section we will review the basic constructions of this theory and add some complements.
We refer to \cite[Section~2.1]{KT} for a brief introduction into the required notions of logarithmic geometry.
The reader unfamiliar with \emph{idealized} log varieties may safely assume that the idealized structure of the considered log varieties is trivial.
We have included the idealized case, as it does not impose any additional challenges and is useful for inductive arguments.
For the benefit of the reader who is not interested in the formalism of logarithmic geometry, we will summarize in Example~\ref{ex:snc} the definition in the case that the log structure is given by a simple normal crossings divisor on a smooth variety.

In order to fix notation, we will always denote the sheaf of monoids defining the logarithmic structure on $X$ by $\logM_X$.
Associated to this sheaf we have the characteristic sheaf $\logMbar_X$ and its associated sheaf of groups $\logMbargp_X$.
The latter is a sheaf of torsion free abelian groups and we will write $r_X$ for its generic rank.
If $X$ is smooth log variety (with trivial idealized structure), then $r_X = 0$.
The rank of $\logMbargp_X$ induces a filtration of $X$ by closed subvarieties
\[ 
    X = X^0 \supseteq X^1 \supseteq \dotsb \supseteq X^{\dim X} \supseteq X^{\dim X + 1} = \emptyset,
\]
where 
\[
    X^i = \{ x \in X : \operatorname{rk}(\logMbargp_{X,x}) \ge i + r_X \}.
\]
If $X$ is idealized, its sheaf of ideals will be denoted by $\logK_X \subseteq \logM_X$.
The underlying ordinary variety of $X$ will be denoted by $\ul X$.

For any closed subset $Z$ of $X$ we define the logarithmic dimension of $Z$ to be
\begin{equation}\label{eq:logdim}
    \logdim Z = \max_{0\le k \le \dim X}\bigl(\dim(Z \cap X^k) + k\bigr) + r_X.
\end{equation}
As $X$ is smooth its sheaf of logarithmic $1$-forms $\Omega^1_X$ is locally free of rank $\logdim X$.
Its dual $\sT_X$ generates the sheaf of logarithmic differential operators $\sD_X$.
We recall from \cite[Corollary~3.4]{KT} the following local description of $\sD_X$.

Fix a closed point $p$ of $X$.
Locally around $p$ we can find a classically smooth morphism $X \to \mathsf{A}_{P,K}$ for $P = \logMbar_p$ and some ideal $K \subseteq P$, where $\mathsf{A}_{P,K} = \Spec k[P]/\langle K \rangle$.
Hence étale locally around $p$ we can identify $X$ with $\as n \times \mathsf{A}_{P,K}$ for some $n$.
Logarithmic coordinates of $X$ at $p$ are given by the coordinates $x_1, \dotsc, x_n$ of $\as n$ (corresponding to elements of $\sO_{X,p}$ giving a system of parameters of $\sO_{S,p}$, where $S$ is the locally closed log stratum through $p$), and by a map $u\colon \logMbar_p \to \sO_{X,p}$ lifting to $\logM_p \to \sO_{X,p}$.

\begin{Proposition}\label{prop:D_X_in_coords}
    Let $X$ be a smooth idealized log variety, $p \in X$ a closed point, and fix logarithmic coordinates of $X$ at $p$ as above.
    Then \'etale locally around $p$ the sheaf $\sD_X$ can be described as the (non-commutative) algebra generated by $x_i$ for $0\leq i\leq n$, by $t^m$ with $m \in \logMbar_{X,p}$, by symbols $\partial_i$ for $0\leq i \leq n$ and $\partial_{m_i}$, where $m_1,\hdots, m_k$ are elements of $\logMbar_p$ that form a basis of $\logMbargp_{X,p}$, subject to all the following relations:
    \begin{align*}
        t^k = 0 & \;\;\;\; \text{ for all } \;\; k \in K_{X,p}, \\
        [x_i,x_j]=0 & \;\;\;\;  \text{ for all } \;\; 0\leq i,j\leq n,\\
        [t^m,t^{m'}]=0 &  \;\;\;\;  \text{ for all }  \;\; m,m' \in \logMbar_{X,p}, \\
        t^{m+m'}=t^m\cdot t^{m'} &  \;\;\;\;  \text{ for all }  \;\; m,m'\in \logMbar_{X,p},\\
        [\partial_i,\partial_j]=0 & \;\;\;\;  \text{ for all } \;\;  1\leq i,j\leq n,\\
        [\partial_i,\partial_{m_j}]=0 & \;\;\;\;  \text{ for all } \;\; 0\leq i\leq n  \; \text{ and } \; 0\leq j\leq k,\\
        [\partial_{m_i},\partial_{m_j}]=0 & \;\;\;\;  \text{ for all } \;\;  1\leq i,j\leq k,\\
        [\partial_i, x_j]=\delta_{ij} & \;\;\;\;  \text{ for all } \;\; 0\leq i,j\leq n, \\
        [\partial_{m_i},t^{m}]= a_i t^m & \;\;\;\;  \text{ for all } \;\; 0\leq i\leq k \; \text{ and } \; m=\, \scriptstyle\sum_i \displaystyle a_i m_i \in \logMbar_{X,p}, \\
        [\partial_i, t^m]=0 & \;\;\;\;  \text{ for all } \;\; 0\leq i\leq n \; \text{ and } \; m\in \logMbar_{X,p},\\
        [\partial_{m_i}, x_j]=0 & \;\;\;\;  \text{ for all } \;\; 0\leq i\leq k  \; \text{ and } \; 0\leq j\leq n.\\
    \end{align*}
\end{Proposition}

If $X$ has trivial idealized structure, then the map $u$ is injective and we can interpret $\partial_{m_i}$ as \enquote{$\partial \log(u(m_i)) = u(m_i) \frac{\partial}{\partial(u_{m_i})}$.}

The sheaf $\sD_X$ has a canonical filtration by the degree of a differential operator.
The associated graded $\gr \sD_X$ is canonical isomorphic to the push-forward $\pi_*\sO_{T^*X}$ of the structure sheaf of the log cotangent space $T^*X$.
As for any filtered sheaf of rings with commutative associated graded, this is used to define the characteristic variety $\Ch(\sheaf M)$ of a coherent $\sD_X$-modules $\sheaf M$ as conical closed subset of the log cotangent space $T^*X$ \cite[Appendix~D]{HottaTakeuchiTanisaki:2008:DModulesPerverseSheavesRepresentationTheory}.
Such $\sheaf M$ is called holonomic if the logarithmic dimension of each component of $\Ch(\sheaf M)$ is equal to $\logdim X$ \cite[Definition~3.22]{KoppensteinerTalpo:2019:HolonomicPerverseLogarithmicDModules}.

\begin{Example}\label{ex:snc}
    A common situation is that the log variety $X$ is given by a simple normal crossings $D$ on a smooth variety $\ul X$.
    In this case we can locally pick coordinates $z_1,\dotsc, z_n$ such that $D$ is given by $z_1\dotsm z_\ell = 0$ for some $0 \le \ell \le n$.
    
    In this situation $X^0 = X$ and $X^i$ is given by the natural stratification of $D$, i.e.
    \[
        X^k = \bigcup_{i_1,\dotsc,i_k} \{ x_{i_1} = \dotsb = x_{i_k} = 0 \}
        \quad\text{ with } i_j \in \{1,\dotsc,\ell\} \text{ distinct.}
    \]
    On each $X^k \setminus X^{k-1}$ the characteristic sheaf $\logMbar_X$ is constant with stalk $\NN^k$ and similarly $\logMbargp_X$ has stalks $\ZZ^k$.
    The idealized structure is trivial, i.e.~$\logK_X = 0$.

    The sheaf of (log) differential operators $\sD_X$ is the subalgebra of the sheaf of (ordinary) differential operators $\sD_{\ul X}$ generated by $z_i\frac{\partial}{\partial z_i}$ for $1 \le i \le \ell$, by $\frac{\partial}{\partial z_i}$ for $\ell < i \le n$, as well as $\sO_X$.
    Thus in particular one has the commutator relation $[z_i\frac{\partial}{\partial z_i}, z_i] = z_i$ for $1 \le i \le \ell$.
    (Log) 1-forms $\Omega_X^1$ are the rank $n$ bundle with sections given by $\frac{dz_i}{z_i}$ for $1 \le i \le \ell$ and $dz_i$ for $i > \ell$.
    Its total space is the (log) cotangent space $T^*X$ (which one should be careful not to confuse with the (ordinary) cotangent space $T^*\ul X$ of $\ul X$).

    The log structure of $T^*X$ is obtained by pulling back the divisor $D$ along the projection $T^*X \to \ul X$.
    The log dimension of closed subsets of $T^*X$ is given as in \eqref{eq:logdim} with $r_{T^*X} = 0$.

    We will sometimes need to consider D-modules on (components) of $X^k$ with the so-called induced idealized log structure.
    This means that $\logMbar_{X^k} = \res{\logMbar_X}{X^k}$, $\sD_{X^k} = \sO_{X^k} \otimes_{\sO_X} \sD_X$, $\Omega_{ X^k}^1 = \sO_{X^k} \otimes_{\sO_X} \Omega_{X}^1$ and $T^*X^k = \res{T^*X}{X^k} = X^k \times_X T^*X$.
    The indexing of the filtration gets shifted so that $(X^k)^i = X^{i+k}$ and one takes $r_{X^k} = r_{T^*X^k} = k$ in \eqref{eq:logdim} (or alternatively one considers $Z$ as a subset of $T^*X$ instead).
\end{Example}

\subsection{Logarithmic D-modules on analytic varieties}

In \cite{KT} one only considered D-modules on complex algebraic varieties.
While this is the main case of interest to us, in order to define the de Rham functor, we need to also consider the corresponding analytic varieties.

For a smooth analytic log variety $X$, one can define the sheaf of rings $\sD_X$ and the various categories of $\sD_X$-modules in an analogous way to the algebraic setting.
As in the classical case there are a few subtle differences between the algebraic and analytic settings.
While $\sD_X$ is still a noetherian coherent sheaf of rings, Lemma~3.8 of \cite{KT} does not hold in the analytic setting and hence coherent $\sD_X$-modules in general only locally admit a good filtration.
As holonomicity is a local condition, one can still talk about the subcategory of holonomic $\sD_X$-modules, and the log Bernstein inequality of \cite[Theorem~3.21]{KT} continues to hold.

In \cite{KT}, duality for log D-modules is defined via a so-called rigid dualizing complex (see Section~\ref{sec:duality} for the concept of a rigid dualizing complex).
In the algebraic setting the global existence of such a complex is guaranteed by the main result of \cite{YekutieliZhang:2006:DualizingComplexesPerverseSheavesOnNoncommutativeRingedSchemes}.
Unfortunately, the proof of this result fails for analytic varieties, as it uses noetherian induction and hence requires the underlying topological space to be noetherian.
However the existence of a rigid dualizing complex is still guaranteed locally, which suffices for most arguments.
(We will also explicitly construct a global rigid dualizing complex in Section~\ref{sec:duality}.)
The log perverse t-structure continues to exist and so does the description of holonomic $\sD_X$-modules given in \cite[Theorem~4.6]{KT}.

If $X$ is an algebraic log variety, we write $X^\an$ for the corresponding analytic log variety \cite[Definition~V.1.1.4]{Ogus:2018:LecturesOnLogarithmicGeometry}.
The canonical map of log spaces
\[
    \iota\colon (X^\an,\,\sO_{X^\an},\,\logM_{X^\an}) \to (X,\,\sO_X\,\logM_X)
\]
is, by definition, strict.
Thus it induces a morphism of sheaves of rings $\iota^{-1}\sD_X \to \sD_{X^\an}$ satisfying
\[
    \sD_{X^\an}
    \cong \sO_{X^\an} \otimes_{\iota^{-1}\sO_X} \iota^{-1}\sD_X
    \cong \iota^{-1}\sD_X \otimes_{\iota^{-1}\sO_X} \sO_{X^\an}.
\]
The functor
\[
    \sheaf M \mapsto \sheaf M^\an =
    \sO_{X^\an} \otimes_{\iota^{-1}\sO_X} \iota^{-1}\sheaf M \cong
    \sD_{X^\an} \otimes_{\iota^{-1}\sD_X} \iota^{-1}\sheaf M
\]
is exact \cite[Proposition~10]{Serre:1956:GAGA}.
Hence it extends to a functor $\catDbDMod{X} \to \catDbDMod{X^\an}$, preserving the coherent and holonomic subcategories.

As noted in the introduction, unless otherwise specified from now on the word variety will include both the complex algebraic and analytic versions.

\subsection{The log canonical bundle}\label{sec:log_canonical}

If $X$ is a smooth idealized log variety, then the log canonical sheaf
\[
    \canon_X = \bigwedge^{\mathclap{\logdim X}} \Omega_X^1
\]
is a line bundle.
It is a right $\sD_X$-module via the negative of the Lie derivative.
Specifically, for a log differential operator $\theta\in \sheaf T_X$, define $(\operatorname{Lie} \theta)\omega$ for $\omega \in \canon_X$ by
\[
((\operatorname{Lie} \theta)\omega) (\theta_1,\hdots, \theta_n)=\theta(\omega(\theta_1,\hdots,\theta_n))-\sum_{i=1}^n \omega(\theta_1,\hdots, [\theta,\theta_i],\hdots, \theta_n)
\]
where $\theta_i\in \sheaf T_X$ and $n=\logdim X$.
Viewing $\theta$ as an element of $\sD_X$, we set $\omega \theta = -(\operatorname{Lie}\theta)\omega$.

\begin{Example}\label{ex:snc_canon}
    If $X = \as n$ with log structure given by the divisor $z_1 \dotsm z_n = 0$, then $\canon_X$ is the trivial line bundle generated by $\frac{dz_1}{z_1} \wedge \dots \wedge \frac{dz_n}{z_n}$.
    The action of $z_i\frac{\partial}{\partial z_i}$ on a section $f\frac{dz_1}{z_1} \wedge \dots \wedge \frac{dz_n}{z_n}$ is given by $-z_i\frac{\partial f}{\partial z_i} \frac{dz_1}{z_1} \wedge \dots \wedge \frac{dz_n}{z_n}$.

    Via the evident map $\sD_X \to \sD_{\ul X}$, the canonical bundle $\canon_{\ul X}$ of the affine space with trivial log structure also is a right $\sD_X$-module.
    There the action of $z_i\frac{\partial}{\partial z_i}$ on a section $f dz_1 \wedge \dots \wedge dz_n$ is given by $-(f + z_i\frac{\partial f}{\partial z_i}) dz_1 \wedge \dots \wedge dz_n$.
    In other words, the natural inclusion
    \[
        \canon_{\ul X} \hookrightarrow \canon_X, \quad
        f dz_1 \wedge \dots \wedge dz_n \mapsto z_1\dotsm z_n f \frac{dz_1}{z_1} \wedge \dots \wedge \frac{dz_n}{z_n}
    \]
    is a morphism of right $\sD_X$-modules.
\end{Example}

For later use, let us study the local situation.
Fix a toric monoid $P$ and consider the affine toric variety $\mathsf{A}_P = \Spec \CC[P]$.
Let $\sigma$ be the convex polyhedral cone generated by $P$ in $P^\gp \otimes \RR$.
For a face $\tau$ of the dual cone $\sigma^\dual$ write $V(\tau)$ for the closure of the corresponding torus orbit $\operatorname{orb}(\tau) = \Spec[P^\gp \cap \tau^\perp]$.
We endow it again with the toric (non-idealized) log structure.
We set $P^\gp_\tau = P^\gp \cap \tau^\perp$.

Set $r = \operatorname{rk} P^\gp$.
By \cite[Proposition~IV.1.1.4]{Ogus:2018:LecturesOnLogarithmicGeometry}, the log canonical bundle $\canon_{\mathsf A_P}$ is identified with $\sO_{\mathsf A_P} \otimes \bigwedge^r P^\gp$.
If $\tau \prec \sigma^\dual$ is a ray (i.e., $V(\tau)$ has codimension $1$) we can write each element of $\canon_{\mathsf A_P}$ in the form $m_1 \wedge \dots \wedge m_r$ with $m_2,\dots,m_r \in P^\gp_\tau$.
Ishida \cite[123]{Ishida:1987:TorusEmbeddingsAndDeRhamComplexes} defines the \emph{Poincar\'e residue map}
\[
    \canon_{\mathsf A_P} = \sO_{\mathsf A_P} \otimes \bigwedge^r P^\gp \to \sO_{V(\tau)} \otimes \bigwedge^{r-1} P^\gp_\tau = \canon_{V(\tau)},
\]
which sends a section $a \otimes m_1 \wedge \dots \wedge m_r$ to $a \otimes \langle m_1, n \rangle m_2 \wedge \dots \wedge m_r$, where $n$ is a primitive element of $(P^\gp)^\dual$ such that $\langle m, n \rangle = 0$ for $m \in P^\gp_\tau$ and $\langle m, n \rangle \ge 0$ for $m \in P^\gp \setminus P^\gp_\tau$.

If we let the differential operator corresponding to $n$ act by zero on $\canon_{V(\tau)}$, the sheaf $\canon_{V(\tau)}$ obtains a right $\sD_{\mathsf A_P}$-module structure.
With the notation of Proposition~\ref{prop:D_X_in_coords}, the action of this differential operator on $\canon_{\mathsf A_P}$ introduces a factor of $t^n$, which vanishes in $\sO_{V(\tau)}$.
It follows that the Poincar\'e residue map is a morphism of right $\sD_{\mathsf A_P}$-modules.

\begin{Example}\label{ex:snc_residue}
    Continuing Example~\ref{ex:snc_canon}, we can identify $\as n$ with $\mathsf A_{\NN^n}$.
    The subvarieties $V(\tau)$, for $\tau \prec \sigma^\dual$ a ray, correspond to the coordinate hyperplanes in $\as n$.
    If $V(\tau) = \{ z_ i = 0 \}$, the Poincar\'e residue map is given by
    \[
        f \frac{dz_1}{z_1} \wedge \dots \wedge \frac{dz_n}{z_n} \mapsto \res{f}{V(\tau)} \frac{dz_1}{z_1} \wedge \dots \wedge \widehat{\frac{dz_i}{dz_i}} \wedge \dots \wedge \frac{dz_n}{z_n}.
    \]
    In particular, if $n = 1$ we obtain the usual residue $f \frac{dz}{z} \mapsto f(0)$.
\end{Example}

\subsection{Left and right modules}\label{sec:left_and_right}

As in the classical situation, one has the following module structures on tensor products and internal homs.

\begin{Proposition}\label{prop:tensor_module_struct}\label{prop:hom_module_struct}
    Let $\sheaf M,\, \sheaf M' \in \catDMod{X}$ and $\sheaf N,\, \sheaf N' \in \catDModOp{X}$.
    Then with $\theta \in \sT_X$ one has the following module structures:
    \begin{align*}
        \sheaf M \otimes_{\sO_X} \sheaf M' &\in \catDMod{X},  & (m \otimes m')\theta & \coloneqq \theta m \otimes m' + m \otimes \theta m', \\
        \sheaf N \otimes_{\sO_X} \sheaf M &\in \catDModOp{X}, & (n \otimes m)\theta & \coloneqq n\theta \otimes m - n \otimes \theta m, \\
        \sheafHom_{\sO_X}(\sheaf M,\, \sheaf M') & \in \catDMod{X},   & (\theta\phi)(m) & \coloneqq \theta(\phi(m)) - \phi(\theta m), \\
        \sheafHom_{\sO_X}(\sheaf N,\, \sheaf N') & \in \catDMod{X},   & (\theta\phi)(n) & \coloneqq -\phi(n)\theta + \phi(n\theta), \\
        \sheafHom_{\sO_X}(\sheaf M,\, \sheaf N)  & \in \catDModOp{X}, & (\phi\theta)(m) & \coloneqq \phi(m)\theta + \phi(\theta m).
    \end{align*}
\end{Proposition}

The following lemma is also classical.

\begin{Lemma}\label{lem:tensor_product_switching}
    Let $\sheaf M,\, \sheaf M' \in \catDMod{X}$ and $\sheaf N \in \catDModOp{X}$.
    Then there exist isomorphisms
    \[
        (\sheaf N \otimes_{\sO_X} \sheaf M') \otimes_{\sD_X} \sheaf M \cong
        \sheaf N \otimes_{\sD_X} (\sheaf M \otimes_{\sO_X} \sheaf M') \cong
        (\sheaf N \otimes_{\sO_X} \sheaf M) \otimes_{\sD_X} \sheaf M'.
    \]
\end{Lemma}

Recall that the log canonical bundle $\canon_X = \bigwedge^{\logdim X} \Omega_X^1$ is a right $\sD_X$-module.
Thus by Proposition~\ref{prop:tensor_module_struct} the assignment
\[
    \sheaf M \mapsto \canon_X \otimes_{\sO_X} \sheaf M
\]
extends to an equivalence of categories $\catDMod{X} \to \catDModOp{X}$ with quasi-inverse given by
\begin{equation}\label{eq:right_to_left}
    \sheaf N \mapsto \canon_X^\dual \otimes_{\sO_X} \sheaf M \coloneqq \sheafHom_{\sO_X}(\canon_X,\, \sheaf M).
\end{equation}
As $\canon_X$ is a line bundle, these operations are clearly exact and extend to the various derived categories.
We will use them frequently to switch between left and right $\sD_X$-modules.

\subsection{Some Resolutions}

Let $X$ be a smooth idealized log variety and fix a coherent left $\sD_X$-module $\sheaf M \in \catCohDMod{X}$.
Assume that $\sheaf M$ has a global good filtration $F_\bullet$.
For any fixed integer $k$ define a complex $\Sp_k(\sheaf M)$ with entries
\[
    \Sp_k^{-p}(\sheaf M) = \sD_X \otimes_{\sO_X} \bigl(\bigwedge^{p}\Theta_X\bigr) \otimes_{\sO_X} F_{k-p}\sheaf M,
    \quad 0 \le p \le \logdim X
\]
and differential
\[
    \delta\colon
    \sD_X \otimes_{\sO_X} \bigl(\bigwedge^{p}\Theta_X\bigr) \otimes_{\sO_X} F_{k-p}\sheaf M \to 
    \sD_X \otimes_{\sO_X} \bigl(\bigwedge^{p-1}\Theta_X\bigr) \otimes_{\sO_X} F_{k-p+1}\sheaf M
\]
given by
\begin{multline*}
    \delta(P \otimes (v_1 \wedge \dotsb \wedge v_p) \otimes m) =
    \sum_{i=1}^p (-1)^{i-1}Pv_i \otimes (v_1 \wedge \dotsb \wedge \hat{v_i} \wedge \dotsb \wedge v_p) \otimes m - \\ -
    \sum_{i=1}^p (-1)^{i-1}P \otimes (v_1 \wedge \dotsb \wedge \hat{v_i} \wedge \dotsb \wedge v_p) \otimes v_im + \\ +
    \sum_{1 \le i \le j \le p} (-1)^{i+j}P \otimes ([v_i,v_j] \wedge v_1 \wedge \dotsb \wedge \hat{v_i} \wedge \dotsb \wedge \hat{v_j} \wedge \dotsb \wedge v_p) \otimes m.
\end{multline*}

\begin{Lemma}[Logarithmic Spencer Resultion]\label{lem:spencer_resolution}
    For any coherent (left) $\sD_X$-module $\sheaf M$ with a fixed good filtration the complex $\Sp_k(\sheaf M)$ is a resolution of $\sheaf M$ for all sufficiently large integers $k$.
\end{Lemma}

This is shown exactly as in the classical case, see for example~\cite[Proposition~I.2.1.18]{Mebkhout:1989:LeFormalismeDesSixOperationsPourLesDModules}.
In particular, one obtains the following very useful resolutions.

\begin{Corollary}\label{cor:canon_resolution}\label{cor:O_resolution}
    One has the following locally free resolutions of the left $\sD_X$-module $\sO_X$ and the right $\sD_X$-module $\canon_X$:
    \begin{equation}\label{eq:O_resolution}
        0 \to \sD_X \otimes_{\sO_X} \bigwedge^{\logdim X} \Theta_X \to \cdots \to \sD_X \otimes_{\sO_X} \bigwedge^0 \Theta_X \to \sO_X \to 0,
    \end{equation}
    \begin{equation}\label{eq:canon_resolution}
        0 \to \bigwedge^{0} \Omega_X \otimes_{\sO_X} \sD_X \to \cdots \to \bigwedge^{\logdim X} \Omega_X \otimes_{\sO_X} \sD_X \to \canon_X \to 0.
    \end{equation}
\end{Corollary}

\begin{proof}
    The complex~\eqref{eq:O_resolution} is just the logarithmic Spencer resolution for $\sO_X$.
    The second line follows from the first by applying the side-switching operation of Section~\ref{sec:left_and_right}, which is evidently exact.
\end{proof}

\section{Operations on D-modules}\label{sec:functors}

In \cite{KT} one only briefly touches on the definition of the pushforward and pullback functors for logarithmic D-modules.
In this section we will go into more detail on the definition of these functors, as well as giving an explicit formula for duality and showing some standard identities for these three operations.

\subsection{Pushforward}

Let $f\colon X \to Y$ be a morphism of smooth idealized log varieties.
If $\sheaf N$ is a (left) $\sD_Y$-module, then as in the classical setting the map $\sT_X \to f^*\sT_Y$ gives the pullback $f^*\sheaf N = \sO_X \otimes_{f^{-1}\sO_Y} f^{-1}\sheaf N$ a canonical $\sD_X$-module structure (see \cite[Section~3.1.4]{KT}).
In particular, one obtains the \emph{transfer module} $\sD_{X \to Y} = f^*\sD_Y$.
As $\sD_Y$ is a locally free $\sO_Y$-module, $\sD_{X \to Y}$ is a locally free $\sO_X$-module concentrated in cohomological degree $0$.

\begin{Lemma}\label{lem:transfer_module_pd}
    Let $f\colon X \to Y$ be a morphism of smooth idealized log varieties.
    Then $\sD_{X \to Y}$ has weak dimension at most $\logdim X$ as left $\sD_X$-module.
\end{Lemma}

\begin{proof}
    Exactness can be checked locally, so we may assume that $X$ is affine and $\sD_{X \to Y}$ is a free $\sO_X$-module, as is $\gr \sD_{X \to Y}$.
    Since locally $\gr \sD_X$ is isomorphic to $\sO_X[t_1,\dotsc,t_{\logdim X}]$, taking a Koszul resolution of the free $\sO_X$-module $\gr_{\sD_{X \to Y}}$ shows that $\gr \sD_{X \to Y}$ has weak dimension at most $\logdim X$ as a $\gr \sD_X$-module (alternatively, we could apply the more general \cite[Lemma~2.3.5]{Bjoerk:1979:RingsOfDifferentialOperators}).
    Thus, by \cite[Proposition~2.3.12]{Bjoerk:1979:RingsOfDifferentialOperators}, $\sD_{X \to Y}$ also has weak dimension at most $\logdim X$ as a left $\sD_X$-module.
\end{proof}

Concretely, as each $F_i\sD_{X \to Y}$ is a locally free $\sO_X$-module, the Spencer complex $\Sp(\sD_{X \to Y})$ is a locally free resolution of $\sD_{X \to Y}$.

As in the classical setting, one defines the $(f^{-1}\sD_Y,\, \sD_X)$-bimodule
\[
    \sD_{Y \from X} = \canon_X \otimes_{\sO_X} \sD_{X \to Y} \otimes_{f^{-1}\sO_Y} f^{-1}\canon_Y^\dual,
\]
which by Lemma~\ref{lem:transfer_module_pd} has finite weak dimension as a right $\sD_X$-module.
Thus the pushforward functor
\[
    f_{\spf} \colon \catDpDMod{X} \to \catDpDMod{Y},\quad \sheaf M \mapsto f_*(\sD_{Y \from X} \otimes_{\sD_X} \sheaf M)
\]
is well defined.
If $f\colon X \to Y$ and $g\colon  Y \to Z$ are two morphisms of smooth idealized log varieties, one can copy the proof of \cite[Proposition~1.5.21]{HottaTakeuchiTanisaki:2008:DModulesPerverseSheavesRepresentationTheory} (which is essentially an application of the projection formula for $f_*\colon \catD{f^{-1}\sD_Y} \to \catD{\sD_Y}$) to show that
\[
    (g \circ f)_\spf \cong g_\spf \circ f_\spf\colon \catDbDMod{X} \to \catDbDMod{Z}.
\]
For a projection $f\colon X \times Y \to Y$ one sets $\Omega_f^k = (\bigwedge^k \Omega_X) \boxtimes \sO_Y$ and defines the \emph{relative de Rham complex} $\DR_f(\sheaf M)$ for $\sheaf M \in \catQCDMod{X \times Y}$ by
\begin{equation}\label{eq:relative_DR_complex}
    (\DR_f(\sheaf M))^i = \begin{cases}
        \Omega_f^{i + \logdim Y} \otimes_{\sO_{X \times Y}} \sheaf M & \text{if } -\logdim Y \le i \le 0 \\
        0 & \text{otherwise}
    \end{cases}
\end{equation}
with the differential as usual induced by the map $\sheaf M \to \sheaf M \otimes \Omega^1_{X \times Y}$ defining the $\sD_{X \times Y}$-module structure on $\sheaf M$.
Since $\sD_{Y \from X \times Y} = \canon_X \boxtimes \sD_Y$ as right $\sD_{X \times Y}$-modules, it follows from Corollary~\ref{cor:canon_resolution} that
\[
    f_{\spf}\sheaf M = f_*\DR_f(\sheaf M).
\]
Any morphism $f\colon X \to Y$ of smooth idealized log varieties can be factored into a closed immersion followed $X \to X \times Y$ followed by the smooth projection $X \times Y \to Y$.
For both of these the pushforward of $\sO$-modules preserves boundedness of complexes.
Thus, by Lemma~\ref{lem:transfer_module_pd}, $f_{\spf}$ also preserves boundedness, i.e.~restricts to a functor $f_{\spf}\colon \catDbDMod{X} \to \catDbDMod{Y}$.

\begin{Lemma}\label{lem:proper_pushforward_preserves_coherence}
    Let $f\colon X \to Y$ be a proper map of smooth idealized log varieties.
    Let $\sheaf M \in \catDbcohDMod{X}$ and assume that locally on $Y$ the cohomology modules of $\sheaf M$ admit a good filtration.
    Then $f_\spf\sheaf M$ is contained in $\catDbcohDMod{Y}$.
\end{Lemma}

\begin{proof}
    Using the resolution of Lemma~\ref{lem:spencer_resolution} one reduces to the fact that pushforward along proper maps preserves coherence of $\sO$-modules, see the first part of \cite[Th\'eor\`eme~I.5.4.1]{Mebkhout:1989:LeFormalismeDesSixOperationsPourLesDModules}.
\end{proof}

Let us remark that here, as in many of the following statements, the condition on the existence of a good filtration is always satisfied in the algebraic setting.

We will call a morphism $i\colon X \to Y$ a \emph{closed immersion} if the underlying morphism of varieties $i\colon \ul X \to \ul Y$ is a closed immersion (note that this is different from the usual use of the term in logarithmic geometry).
For a closed immersion $i$, we can give the right adjoint to $i_\spf$ in the usual way:

\begin{Definition}\label{def:!-pullback_closed_immersion}
    Let $i\colon Z \hookrightarrow X$ be a closed immersion of smooth idealized log varieties.
    Define $i^!\colon \catDpDMod{X} \to \catDpDMod{X}$ by
    \[
        \sheaf M \mapsto \sheafHom_{i^{-1}\sD_X}(\sD_{X \from Z},\, i^{-1}\sheaf M),
    \]
    where the left $\sD_Z$-module structure is given by the right action of $\sD_Z$ on $\sD_{X \from Z}$.
\end{Definition}

\begin{Proposition}\label{prop:i^!_adjunction}
    Let $i\colon Z \hookrightarrow X$ be a closed immersion of smooth idealized log varieties.
    Then $i^!$ is the right adjoint to $i_\spf \colon \catDpDMod{Z} \to \catDpDMod{X}$.
\end{Proposition}

\begin{Lemma}\label{lem:i^!_and_Gamma}
    Let $i\colon Z \hookrightarrow X$ be a closed immersion of smooth idealized log varieties.
    Then for any $\sheaf M \in \catDpDMod{X}$ there exists a canonical equivalence
    \[
        \Hom_{i^{-1}\sD_X}(\sD_{X \from Z},\, i^{-1}\sheaf M)
        \cong
        \Hom_{i^{-1}\sD_X}(\sD_{X \from Z},\, i^{-1}\Gamma_Z\sheaf M).
    \]
\end{Lemma}

\begin{proof}
    We first show the statement on the level of abelian categories, i.e.~for $\sheaf M \in \catDMod{X}$.
    Let $\sheaf J$ be the sheaf of ideals defining the closed subvariety $Z$.
    Then $i^{-1}\sheaf J$ annihilates $\sD_{X \from Z}$ and for any $i^{-1}\sD_X$-module map $\psi\colon \sD_{X \from Z} \to \sheaf M$ one has $(i^{-1}\sheaf J)\psi(m) = \psi((i^{-1}\sheaf J)m) = 0$, i.e.~$\psi$ factors through $i^{-1}R^0\Gamma_Z(\sheaf M)$.

    To deduce the derived statement, it is now sufficient to show that if $\sheaf I$ is an injective $\sD_Y$-module, then $i^{-1}R^0\Gamma_Z \sheaf I$ is an injective $i^{-1}\sD_X$-module.
    This follows from
    \begin{align*}
        R^0\Hom_{i^{-1}\sD_X}(\sheaf N,\, R^0i^{-1}\Gamma_Z \sheaf I) & \cong
        R^0\Hom_{i^{-1}\sD_X}(i^{-1}i_*\sheaf N,\, i^{-1}R^0\Gamma_Z \sheaf I) \\ & \cong
        R^0\Hom_{\sD_X}(i_*\sheaf N,\, R^0\Gamma_Z \sheaf I) \\ & \cong
        R^0\Hom_{\sD_X}(i_*\sheaf N,\, \sheaf I)
    \end{align*}
    for any $\sD_Z$-module $\sheaf N$.
\end{proof}

\begin{proof}[Proof of Proposition~\ref{prop:i^!_adjunction}]
    By Lemma~\ref{lem:i^!_and_Gamma} and tensor-hom adjunction we have
    \begin{align*}
        \Hom_{\sD_Z}(\sheaf M,\, i^!\sheaf N) & =
        \Hom_{\sD_Z}\bigr(\sheaf M,\, \sheafHom_{i^{-1}\sD_X}(\sD_{X \from Z},\, i^{-1}\sheaf N)\bigl) \\ & =
        \Hom_{\sD_Z}\bigr(\sheaf M,\, \sheafHom_{i^{-1}\sD_X}(\sD_{X \from Z},\, i^{-1}\Gamma_Z\sheaf N)\bigl) \\ & =
        \Hom_{i^{-1}\sD_X}\bigr(\sD_{X \from Z} \otimes_{\sD_Z} \sheaf M,\, i^{-1}\Gamma_Z\sheaf N\bigl) \\ & =
        \Hom_{\sD_X}\bigr(i_*(\sD_{X \from Z} \otimes_{\sD_Z} \sheaf M),\, \sheaf N\bigl) \\ & =
        \Hom_{\sD_X}(i_\spf\sheaf M,\, \sheaf N).\qedhere
    \end{align*}
\end{proof}

\subsection{Duality}\label{sec:duality}

Let $X$ be a (algebraic of analytic) scheme together with an $\sO_X$-ring $\sheaf A$, that is, a sheaf of (not necessarily commutative) noetherian rings $\sheaf A$ on $X$ endowed with a ring homomorphism $\sO_X \to \sheaf A$.
A dualizing complex over $\sheaf A$ is a bounded complex $\sheaf R$ of $\sheaf A$-bimodules such that the following conditions hold:
\begin{enumerate}
    \item The functors
        \[
            D\colon \catDb{\sheaf A}^{\opcat} \to \catD{\sheaf A^{\opring}}, \quad \sheaf M \mapsto \sheafHom_{\sheaf A}(\sheaf M, \sheaf R)
        \]
        and
        \[
            D^{\opring}\colon \catDb{\sheaf A^\opring}^\opcat \to \catD{\sheaf A}, \quad \sheaf M \mapsto \sheafHom_{\sheaf A^{\opring}}(\sheaf M, \sheaf R)
        \]
        have finite cohomological dimension when restricted to $\cat{Coh}(\sheaf A)$ and $\cat{Coh}(\sheaf A^\opring)$ respectively.
    \item
        The functors $D$ and $D'$ preserve coherence.
    \item
        The adjunction morphisms $\id \to D^\opring D$ in $\catDb[\mathrm{coh}]{\sheaf A}$ and $\id \to D D^\opring$ in $\catDb[\mathrm{coh}]{\sheaf A^\opring}$ are both isomorphisms.
\end{enumerate}
Unfortunately, dualizing complexes are not necessarily unique.
For example, in the commutative case $\sheaf A = \sO_X$, they are only determined up to shift and tensoring with a line bundle.
To restore uniqueness and obtain a duality theory in the non-commutative setting, van den Bergh \cite{VanDenBergh:1997:ExistenceTheoremsForDualizingComplexes} introduced the concept of a \emph{rigid dualizing complex}.
This is a dualizing complex $\sheaf R$ together with a rigidifying isomorphism\footnote{The definition presented here is stricter than actually needed for a general theory, but suffices for our purposes. We refer to \cite[Sections~2 and~3]{YekutieliZhang:2006:DualizingComplexesPerverseSheavesOnNoncommutativeRingedSchemes} for the general definition.}
\[
    \rho\colon \sheaf R \isoto \sheafHom_{\sheaf A \otimes \sheaf A^\opring}( \sheaf A,\, \sheaf R \otimes \sheaf R ).
\]
If it exists, a rigid dualizing complex is unique up to unique isomorphism.

Let us now return to the case that $X$ is a smooth idealized log variety and $\sheaf A = \sD_X$.
As was discussed in \cite{KT}, if $X$ is algebraic, then $\sD_X$ always admits a rigid dualizing complex $\rigidDC$.
As we will see below, if the log structure of $X$ is non-trivial, $\rigidDC$ is not simply a shift of $\sD_X$.

To show that the de Rham functor commutes with duality, we need an explicit computation of $\rigidDC$.
In the case that $\ul X$ is smooth this is a special case of \cite{Chemla:2004:RigidDualizingComplexForQuantumEnvelopingAlgebras} where the rigid dualizing complex of the enveloping algebra of a Lie algebroid is computed.
In this case the canonical bundle $\canon_{\ul X}$ is a right $\sD_{\ul X}$-module and hence by restriction also a right $\sD_X$-module.
The rigid dualizing complex of $\sD_X$ is then given by $\sD_X \otimes_{\sO_X} \sheafHom_{\sO_X}(\canon_X,\, \canon_{\ul X})[\dim X + \logdim X]$.

We will show that this formula holds for general smooth idealized log varieties.
To do so, one replaces the canonical bundle $\canon_{\ul X}$ by the Grothendieck dualizing complex $\Odc_{\ul X}$, i.e.~if $p\colon \ul X \to \pt$ is the structure map, then $\Odc_{\ul X} = p_{\sO\text{-mod}}^!\CC$.
Hence, if $\ul X$ is smooth, one has $\Odc_{\ul X} = \canon_{\ul X}[\dim X]$.

To give $\Odc_{\ul X}$ the structure of a right $\sD_X$-module we can use a result of Tsuji \cite[Theorem~2.21]{Tsuji:1999:PoincareDualityForLogCrystallineCohomology}.
For comparison with duality on the Kato--Nakayama space, an explicit description of $\Odc_{\ul X}$ given by Ishida is particularly useful.
We will sketch the construction of Ishida's complex.
A detailed translation of Ishida's construction into the language of logarithmic geometry is given in \cite[Section~1]{Fornasiero:2006:IshidaAndDuBoisComplexes}, while the local (toric) situation is discussed in \cite[Section~3.2]{Oda:ConvexBodiesAndAlgebraicGeoemtry}

Any smooth idealized log variety $X$ is \'etale locally isomorphic to a union of orbits of a normal toric variety.
Thus the filtration $X^k$ turns $X$ into a filtered semi-toroidal variety on the sense of \cite[Definiton~5.2]{Ishida:1987:TorusEmbeddingsAndDeRhamComplexes}.
Let $\phi_k\colon \tilde X^k \to X^k $ be the normalization of the closed log stratum $X^k$.
Since $X$ is \'etale locally isomorphic to a normal toric variety with its toric boundary divisor, $\tilde X^k$ is the disjoint union of the irreducible components of $X^k$.
As each component is again \'etale locally isomorphic to a toric variety, $\tilde X^k$ has a canonical induced smooth log structure (note that this log structure is not idealized).
Set
\[
    \sheaf C^i = \phi_{i,*}\canon_{\tilde X^{i}}, \quad 0 \le i \le \dim X.
\]
The local Poincar\'e residue maps (see Section~\ref{sec:log_canonical}) induce a differential $d\colon \sheaf C^i \to \sheaf C^{i+1}$, making $(\sheaf C^\bullet,\, d)$ into a complex of $\sO_X$-modules.

\begin{Proposition}[{\cite[Theorem~5.4]{Ishida:1987:TorusEmbeddingsAndDeRhamComplexes}}]\label{prop:Ishida_canonical}
    The complex $(\sheaf C^\bullet,\, d)$ is isomorphic to $\Odc_{\ul X}[-\dim X]$.
\end{Proposition}

\begin{Remark}
    Any smooth log variety is \'etale locally isomorphic to an affine toric variety and hence is Cohen--Macaulay.
    Thus despite the above description of $\Odc_{\ul X}$ as a complex, in this case it is actually a sheaf concentrated in the single degree $-\dim \ul X$.
    In other words the complex $\sheaf C^\bullet$ only has non-trivial cohomology at $\sheaf C^0$.
\end{Remark}

\begin{Example}\label{ex:snc_ishida}
    Let us again consider the example of the affine space $X = \as n$ with the log structure given by the coordinate hyperplanes.
    Set $Z_i = \{z_i = 0\}$.
    The first differential in the complex is then
    \[
        \canon_X \to \bigoplus_{i=1}^n \canon_{Z_i}, \quad
        f\frac{dz_1}{z_1} \wedge \dots \wedge \frac{dz_n}{z_n} \mapsto
        \sum_{i=1}^n \res{f}{Z_i} \frac{dz_1}{z_1} \wedge \dots \wedge \widehat{\frac{dz_i}{z_i}} \wedge \dotsm \wedge \frac{dz_n}{z_n}.
    \]
    The kernel of this morphism consists of all sections of the form $z_1\dotsm z_n f\frac{dz_1}{z_1} \wedge \dots \wedge \frac{dz_n}{z_n}$, agreeing with Example~\ref{ex:snc_canon}.
\end{Example}

Each $\canon_{\tilde X^k}$ is canonically a right $\sD_{\tilde X^i}$-module and this structure trivially extends to a  right $\sD_X$-module structure on $\sheaf C^i$, which one checks to be compatible with the differential $d$.
Thus we have defined a right $\sD_X$-module structure on $\Odc_{\ul X}$.
If $\ul X$ is smooth then $\canon_{\ul X}$ is a submodule of $\canon_X$ and hence this definition agrees with the $\sD_X$-module structure on $\canon_{\ul X}$ given by restriction of the Lie derivative.
We let $\twist_X$ be the left module corresponding by \eqref{eq:right_to_left} to $\Odc_{\ul X}$, that is,
\[
    \twist_X = \sheafHom_{\sO_X}(\canon_X,\,\Odc_{\ul X}).
\]

From Proposition~\ref{prop:Ishida_canonical} one obtains the following corollary.

\begin{Corollary}\label{cor:complex_for_twist}
    \[
        \twist_X \cong \bigl(\phi_{0,*}\sO_{\tilde X^0} \to \dotsc \to \phi_{\dim X,*}\sO_{\tilde X^{\dim X}}\bigr)[\dim X].
    \]
\end{Corollary}

\begin{Example}\label{ex:snc_twist}
    Continuing Examples~\ref{ex:snc_canon} and \ref{ex:snc_ishida}, $\twist_{\as n}$ is the left $\sD_{\as n}$-module $z_1\dotsm z_n \sO_{\as n}[n]$ and the above corollary presents this as the kernel of the restriction morphism $\sO_{\as n} \mapsto \bigoplus_i \sO_{Z_i}$.
\end{Example}

\begin{Remark}\label{rem:why_twist}
    One notes that if $X = \ul X$ is a smooth variety, then $\twist_X = \sO_X[\dim X]$.
    The general idea is that all occurrences of this shifted structure sheaf in the classical theory should be replaced by $\twist_X$ is the logarithmic theory.
    In particular, all shifts by $\dim X$ should replaced by tensoring with $\twist_X$.
\end{Remark}

Following this idea, the dualizing complex for $\sD_X$ should be given by $\sD_X \otimes_{\sO_X} \twist_X$.
There are two possible left $\sD_X$-module structures on this complex: one by left multiplication and one by Proposition~\ref{prop:tensor_module_struct} as the tensor product of two left modules.
Further, $\sD_X \otimes_{\sO_X} \twist_X$ has a right module structure induced from the right module $\sD_X$-module structure on $\sD_X$ and Proposition~\ref{prop:tensor_module_struct}.
Unless otherwise stated, we will view $\sD_X \otimes_{\sO_X} \twist_X$ as a $\sD_X \otimes \sD_X^\opring$-module via left multiplication and the tensor product right module structure.

\begin{Lemma}\label{lem:free_module_tensor_involution}\label{lem:dualizing_complex_involution}
    Let $Z$ be a component of some closed log stratum $X^k$ with the induced idealized log structure and let $\sheaf M$ be a left $\sD_Z$-module which is locally free as an $\sO_Z$-module.
    Then there exists an involution of $\sD_X \otimes_{\sO_X} \sheaf M$ that interchanges the $\sD_X$-module structures given by left multiplication with that of Proposition~\ref{prop:tensor_module_struct}, and fixes the submodule $\sheaf M$.
    In particular there exists such an involution of $\sD_X \otimes_{\sO_X} \twist_X$.
    A similar statement holds for right modules.
\end{Lemma}

\begin{proof}
    As $\sD_X \otimes_{\sO_X} \sheaf M = \sD_Z \otimes_{\sO_Z} \sheaf M$, we can assume that $Z = X$.
    Let $\sheaf N$ and $\sheaf N'$ be $\sD_X \otimes_{\sO_X} \sheaf M$ with the two different left $\sD_X$-actions.
    Then $\sheaf N$ (resp.~$\sheaf N'$) have good filtrations $F_\bullet$ (resp.~$F'_\bullet$) such that $F_0\sheaf N = F'_0\sheaf N'$ and $\gr_{F}\sheaf N = \gr_{F'}\sheaf N'$.
    Lifting the identity morphism of this associated graded we obtain the desired involution (see~\cite[Lemma~D.2.3]{HottaTakeuchiTanisaki:2008:DModulesPerverseSheavesRepresentationTheory}).

    The involution of $\sD_X \otimes_{\sO_X} \twist_X$ follows from the description of $\Odc_{\ul X}$ given in Proposition~\ref{prop:Ishida_canonical} and the fact that the involutions fix $\Hom_{\sO_X}(\canon_X,\  \sheaf C^i)$.
\end{proof}

\begin{Lemma}
    The complex $\sD_X \otimes_{\sO_X} \twist_X$ has finite injective dimension.
\end{Lemma}

\begin{proof}
    It suffices to show that for sufficiently large $i$ one has $\Ext^i(\sheaf M,\, \sD_X \otimes_{\sO_X} \twist_X) = 0$ for all $\sD_X$-modules $\sheaf M$.
    One can further restrict to $\sheaf M = \sD_X / I$ for all left ideals $I$ of $\sD_X$.
    Such $\sheaf M$ are clearly coherent and endowed with a global good filtration.
    The filtration on $\sD_X$ induces a good filtration on $\sD_X \otimes_{\sO_X} \twist_X$ such that its associated graded is given by $\pi^*(\canon_X^\dual \otimes_{\sO_X} \Odc_{\ul X})$, where $\pi \colon T^*X \to X$ is the projection from the log cotangent bundle.
    As this is a dualizing complex, it follows that $\Ext^i\bigl(\gr \sheaf M,\, \gr(\sD_X \otimes_{\sD_X} \twist_X)\bigr)$ vanishes for all sufficiently large $i$ independently of $\sheaf M$.
    Thus also $\Ext^i(\sheaf M,\, \sD_X \otimes_{\sO_X} \twist_X)$ vanishes.
\end{proof}

For any left $\sD_X$-module $\sheaf M$ the $\sD_X^\opring$-structure on $\sD_X \otimes_{\sO_X} \twist_X$ induces a right module structure on $\sheafHom_{\sD_X}(\sheaf M,\, \sD_X \otimes_{\sO_X} \twist_X)$.
Thus we define the duality functor
\begin{align*}
    \DVerdier_X \colon & \catDbcohDMod{X}^\opcat \to \catDbcohDMod{X} \\
                       & \sheaf M \mapsto \sheafHom_{\sD_X}(\sheaf M,\, \sD_X \otimes_{\sO_X} \twist_X) \otimes_{\sO_X} \canon_X^\dual.
\end{align*}
We will show in Theorem~\ref{thm:rigid_dualizing_complex} that, up to a shift, $\sD_X \otimes_{\sO_X} \twist_X$ is indeed the rigid dualizing complex for $\sD_X$.
In particular, we will have that $\DVerdier_X \circ \DVerdier_X \cong \id$.
The proof of Theorem~\ref{thm:rigid_dualizing_complex} will be based on the following duality theorem.

\begin{Theorem}\label{thm:proper_duality}
    Let $f\colon X \to Y$ be a morphism of smooth idealized log varieties such that the induced map $f\colon \ul X \to \ul Y$ on the underlying varieties is projective.
    Let $\sheaf M \in \catDbcohDMod{X}$ and assume that locally on $Y$ the cohomology modules of $\sheaf M$ admit a good filtration.
    Then there exists a canonical isomorphism in $\catDbcohDMod{Y}$
    \[
        f_\spf \circ \DVerdier_X \sheaf M \cong \DVerdier_Y \circ f_\spf \sheaf M.
    \]
\end{Theorem}

We note that by Lemma~\ref{lem:proper_pushforward_preserves_coherence} $f_\spf \sheaf M$ is indeed coherent, so that the statement of the theorem makes sense.

\begin{Lemma}\label{lem:i^!_of_twist}
    Let $i\colon Z \to X$ be a closed immersion of smooth idealized log varieties.
    Then $i^!(\sD_X \otimes_{\sO_X} \twist_X) \cong i^*\sD_X \otimes_{\sO_Z} \twist_Z$.
\end{Lemma}

\begin{proof}
    \begin{align*}
        i^!\twist_X &= \sheafHom_{i^{-1}\sD_X}(\sD_{X \from Z},\, i^{-1}\sD_X \otimes_{i^{-1}\sO_X} i^{-1}\twist_X) \\ & \cong
        \sheafHom_{i^{-1}\sD_X}\bigl(i^{-1}\sD_X \otimes_{i^{-1}\sO_X} i^{-1}\canon_X^\dual \otimes_{i^{-1}\sO_X} \canon_Z,\, i^{-1}\sD_X \otimes_{i^{-1}\sO_X} i^{-1}\twist_X \bigr) \\ & \cong
        \sheafHom_{i^{-1}\sO_X}\bigl(i^{-1}\canon_X^\dual \otimes_{i^{-1}\sO_X} \canon_Z,\, i^{-1}\sD_X \otimes_{i^{-1}\sO_X} i^{-1}\twist_X \bigr) \\ & \cong
        i^{-1}\sD_X \otimes_{i^{-1}\sO_X} \sheafHom_{i^{-1}\sO_X}\bigl(i^{-1}\canon_X^\dual \otimes_{i^{-1}\sO_X} \canon_Z,\, i^{-1}\canon_X^\dual \otimes_{i^{-1}\sO_X} i^{-1}\Odc_{\ul X} \bigr) \\ & \cong
        i^*\sD_X \otimes_{\sO_Z} \sheafHom_{i^{-1}\sO_X}\bigl(\canon_Z,\, i^{-1}\Odc_{\ul X} \bigr) \\ & \cong
        i^*\sD_X \otimes_{\sO_Z} \sheafHom_{i^{-1}\sO_X}\bigl(\sO_Z \otimes_{\sO_Z} \canon_Z,\, i^{-1}\Odc_{\ul X} \bigr) \\  & \cong
        i^*\sD_X \otimes_{\sO_Z} \sheafHom_{\sO_Z}\bigl(\canon_Z,\, \sheafHom_{i^{-1}\sO_X}(\sO_Z, i^{-1}\Odc_{\ul X}) \bigr) \\ & \cong
        i^*\sD_X \otimes_{\sO_Z} \sheafHom_{\sO_Z}\bigl(\canon_Z,\, \Odc_{\ul Z} \bigr) \cong 
        i^*\sD_X \otimes_{\sO_Z} \twist_Z. \qedhere
    \end{align*}
\end{proof}

\begin{Lemma}\label{lem:proper_duality_morphism}
    Let $f\colon X \to Y$ be a morphism of smooth idealized log varieties such that the induced map $f\colon \ul X \to \ul Y$ on the underlying varieties is projective.
    Then there exists a canonical morphism of functors $\catDbcohDMod{X} \to \catDbcohDMod{Y}$
    \[
        f_\spf \circ \DVerdier_X \to \DVerdier_Y \circ f_\spf.
    \]
\end{Lemma}

\begin{proof}
    For any $\sheaf M \in \catDbcohDMod{X}$ we have
    \begin{align*}
        f_\spf\DVerdier_X\sheaf M & \cong f_*\bigl(\sheafHom_{\sD_X}(\sheaf M,\, \sD_X \otimes_{\sO_X} \twist_X) \otimes_{\sD_X} \sD_{X \to Y} \bigr) \otimes_{\sO_X} \canon_Y^\dual \\
                                  & \cong f_*\bigl(\sheafHom_{\sD_X}(\sheaf M,\, \sD_{X\to Y} \otimes_{\sO_X} \twist_X) \bigr) \otimes_{\sO_X} \canon_Y^\dual
        \intertext{and}
        \DVerdier_X f_\spf \sheaf M & \cong \sheafHom_{\sD_Y}(f_\spf\sheaf M,\, \sD_Y \otimes_{\sO_Y} \twist_Y) \otimes_{\sO_Y} \canon_Y^\dual.
    \end{align*}
    Considering the morphism
    \begin{align*}
        f_*\sheafHom_{\sD_X}(\sheaf M,\, \sD_{X\to Y} \otimes_{\sO_X} \twist_X)
        & \to
        f_*\sheafHom_{f^{-1}\sD_Y}(\sD_{Y \from X} \otimes_{\sD_X} \sheaf M,\, \sD_{Y \from X} \otimes_{\sD_X} \sD_{X\to Y} \otimes_{\sO_X} \twist_X)
        \\ & \to
        \sheafHom_{\sD_Y}\bigl(f_*(\sD_{Y \from X} \otimes_{\sD_X} \sheaf M),\, f_*(\sD_{Y \from X} \otimes_{\sD_X} \sD_{X\to Y} \otimes_{\sO_X} \twist_X)\bigr)
        \\ & \cong
        \sheafHom_{\sD_Y}\bigl(f_\spf(\sheaf M),\, f_\spf(\sD_{X\to Y} \otimes_{\sO_X} \twist_X)\bigr),
    \end{align*}
    it thus suffices to construct a canonical morphism
    \[
        f_\spf(\sD_{X \to Y} \otimes_{\sO_X} \twist_X) \to \sD_Y \otimes_{\sO_Y} \twist_Y.
    \]
    We can assume that $f$ is either a closed immersion $X \to \ul{\ps n} \times Y$ or the projection $\ul{\ps n} \times Y \to Y$.
    In the first case, by Proposition~\ref{prop:i^!_adjunction} and Lemma~\ref{lem:i^!_of_twist} we have
    \begin{align*}
        \Hom_{\sD_Y}(f_\spf(\sD_{X \to Y} \otimes_{\sO_X} \twist_X),\, \sD_Y \otimes_{\sO_Y} \twist_Y) &\cong
        \Hom_{\sD_Y}(f^*\sD_Y \otimes_{\sO_X} \twist_X,\, f^!(\sD_Y \otimes_{\sO_Y} \twist_Y)) \\&\cong
        \Hom_{\sD_Y}(f^*\sD_Y \otimes_{\sO_X} \twist_X,\, f^*\sD_Y \otimes_{\sO_X} \twist_X).
    \end{align*}
    Thus we take the image of the identity morphism for our map.
    
    If $f$ is a projection $\ulps n \times Y \to Y$, then $\twist_{\ulps n \times Y} = \twist_{\ulps n} \boxtimes \twist_Y$.
    Thus we can reduce to the case that $Y$ is a point.
    But $\twist_{\ulps n} = \sO_{\ps n}[n]$, so that the desired morphism is just the classical trace morphism $f_{\spf}\sO_{\ps n}[n] \to \CC$.

    Finally, we need to check that the constructed morphism is independent of the chosen factorization of $f$.
    Given two different factorizations as above one can form a commutative diagram
    \[
        \begin{tikzcd}
            & \ulps{n_1} \times Y \arrow[dr, hook, "j_1"] \arrow[d] & \\
            X \arrow[ur, hook, "i_1"] \arrow[r, "f"] \arrow[dr, hook, "i_2"] & Y & \ulps N \times Y \arrow[l] \\
            & \ulps{n_2} \times Y \arrow[ur, hook, "j_2"] \arrow[u] & 
        \end{tikzcd}
    \]
    for some $N \ge n_1, n_2$.
    For closed immersions, the constructed map is obtained by adjunction.
    Hence it is compatible with composition.
    Thus it suffices to show that $\ulps{n_1} \times Y \to Y$ and $\ulps{n_1} \times Y \to \ulps{N} \times Y \to Y$ induce the same duality map.
    As before, we reduce to the case that $Y$ is a point.
    But then all log structures are trivial, so that the statement is just usual duality for D-modules.
\end{proof}

\begin{proof}[Proof of Theorem~\ref{thm:proper_duality}]
    We have to show that the morphism of Lemma~\ref{lem:proper_duality_morphism} is an isomorphism.
    As usual, it suffices to show this for $\sheaf M$ an object of the abelian category $\catCohDMod{X}$.
    Then by assumption, locally on $Y$, $\sheaf M$ is generated by an $\sO_X$-coherent $\sO_X$-submodule, i.e.~it is a quotient of a module of the form $\sD_X \otimes_{\sO_X} \sheaf G$ for some coherent $\sO_X$-module $\sheaf G$, with action given by left multiplication.
    Thus by the Way-out Lemma it suffices to prove the statement for modules of the form $\sD_X \otimes_{\sO_X} \sheaf G$.

    We thus compute
    \begin{align*}
        f_\spf\DVerdier_X(\sD_X \otimes_{\sO_X} \sheaf G) &=
        f_\spf\bigl(\sheafHom_{\sD_X}(\sD_X \otimes_{\sO_X} \sheaf G,\, \sD_X \otimes_{\sO_X} \twist_X) \otimes_{\sO_X} \canon_X^\dual \bigr) \\ &=
        f_\spf\bigl(\sD_X \otimes_{\sO_X} \sheafHom_{\sO_X}( \sheaf G,\, \sO_X) \otimes_{\sO_X} \twist_X \otimes_{\sO_X} \canon_X^\dual\bigr) \\ &=
        f_*\bigl(\sD_{Y \from X} \otimes_{\sO_X} \sheafHom_{\sO_X}( \sheaf G,\, \twist_X) \otimes_{\sO_X} \canon_X\bigr) \\ &=
        f_*\bigl(f^{-1}(\sD_Y \otimes_{\sO_Y} \canon_Y^\dual) \otimes_{f^{-1}\sO_Y} \sheafHom_{\sO_X}(\sheaf G,\, \twist_X)\bigr) \\ &=
        \sD_Y \otimes_{\sO_Y} \canon_Y^\dual \otimes_{\sO_Y} f_*\sheafHom_{\sO_X}(\sheaf G \otimes_{\sO_X} \canon_X,\, \Odc_{\ul X})
        \intertext{and}
        \DVerdier_Yf_\spf(\sD_X \otimes_{\sO_X} \sheaf G) & = 
        \DVerdier_Y\bigl(\sD_Y \otimes_{\sO_Y} \canon_Y^\dual \otimes_{\sO_X} f_*(\canon_X \otimes_{\sO_X} \sheaf G) \bigr) \\ & = 
        \sheafHom_{\sD_Y}\bigl(\sD_Y \otimes_{\sO_Y} \canon_Y^\dual \otimes_{\sO_X} f_*(\canon_X \otimes_{\sO_X} \sheaf G),\, \sD_Y \otimes_{\sO_Y} \twist_Y \bigr) \otimes_{\sO_Y} \canon_Y^\dual \\ & = 
        \sD_Y \otimes_{\sO_Y} \sheafHom_{\sO_Y}( f_*(\canon_X \otimes_{\sO_X} \sheaf G),\, \twist_Y) \\ & = 
        \sD_Y \otimes_{\sO_Y} \sheafHom_{\sO_Y}( f_*(\canon_X \otimes_{\sO_X} \sheaf G),\, \Odc_{\ul Y}) \otimes_{\sO_Y} \canon_{Y}^\dual.
    \end{align*}
    Thus the result follows from Grothendieck duality
    \[
        f_*\sheafHom_{\sO_X}(\sheaf G \otimes_{\sO_X} \canon_X,\, \Odc_{\ul X})
        \cong
        \sheafHom_{\sO_Y}( f_*(\canon_X \otimes_{\sO_X} \sheaf G),\, \Odc_{\ul Y}).
        \qedhere
    \]
\end{proof}

\begin{Theorem}\label{thm:rigid_dualizing_complex}
    Let $X$ be a smooth idealized log variety.
    Then the rigid dualizing complex for $\sD_X$ is given by $\sD_X \otimes_{\sO_X} \twist_X[\logdim X]$.
\end{Theorem}

\begin{proof}
    We already know that $\sD_X \otimes_{\sO_X} \twist_X$ has finite injective dimension.
    Showing that the canonical map
    \[
       \sheaf M \to \sheafHom_{\sD_X^\opring}\bigl(\sheafHom_{\sD_X}(\sheaf M,\, \sD_X \otimes_{\sO_X} \twist_X),\, \sD_X \otimes_{\sO_X} \twist_X\bigr)
    \]
    is an isomorphism for coherent $\sheaf M$ is a local question.
    Hence we can assume that $\sheaf M$ has a free resolution.
    Inducting on the length of the resolution, we can further reduce to the case that $\sheaf M = \sD_X$.
    Then,
    \begin{multline*}
        \sheafHom_{\sD_X^\opring}\bigl(\sheafHom_{\sD_X}(\sD_X,\, \sD_X \otimes_{\sO_X} \twist_X),\, \sD_X \otimes_{\sO_X} \twist_X\bigr) \cong \\
        \sheafHom_{\sD_X^\opring}\bigl(\sD_X \otimes_{\sO_X} \twist_X,\, \sD_X \otimes_{\sO_X} \twist_X\bigr) \cong
        \sheafHom_{\sD_X^\opring}\bigl(\sD_X,\, \sD_X\bigr) \cong
        \sD_X,
    \end{multline*}
    where the second isomorphism follows via tensor-hom adjunction from Grothendieck duality $\Hom_{\sO_X}(\twist_X,\, \twist_X) \cong \sO_X$.

    Let $\Delta \colon X \to X \times X$ be the diagonal map.
    Then for any right $\sD_X$-module $\sheaf M$ we have $\sheaf M \otimes_{\sD_X} \sD_{X \to X \times X} \cong \sheaf M \otimes_{\sO_X} \sD_X$ as right $f^{-1}\sD_{X \times X}$-modules.
    It follows from the resolution \eqref{eq:canon_resolution} that $\DVerdier_X \sO_X = \twist_X [-\logdim X]$.
    Thus $\Delta_\spf \twist_X = \DVerdier_{X\times X} \Delta_\spf \sO_X[\logdim X]$.
    Hence, with $d = \logdim X$ and applying Lemma~\ref{lem:free_module_tensor_involution} we have
    \begin{multline*}
        \Delta_*(\Odc_{\ul X} \otimes_{\sO_X} \sD_X) \otimes_{\sO_X} \canon_{X \times X}^\dual =
        \Delta_\spf\twist_X
        \\ \cong
        \sheafHom_{\sD_{X \times X}}\bigl(\Delta_*(\canon_X \otimes_{\sO_X} \sD_X) \otimes_{\sO_{X \times X}} \canon_{X\times X}^\dual,\, \sD_{X \times X} \otimes_{\sO_{X \times X}} \twist_{X\times X} \otimes_{\sO_{X \times X}} \canon_{X \times X}^\dual\bigr)[d]
        \\ \cong
        \sheafHom_{\sD_X \otimes \sD_X}\bigl(\sD_X \otimes_{\sO_X} \canon_X^\dual,\, (\sD_X \otimes_{\sO_X} \twist_X \otimes_{\sO_X} \canon_X^\dual) \otimes (\sD_X \otimes_{\sO_X} \twist_X \otimes_{\sO_X} \canon_X^\dual)\bigr)[d]
        \\ \cong
        \sheafHom_{\sD_X \otimes \sD_X^\opring}\bigl(\sD_X,\, (\sD_X \otimes_{\sO_X} \twist_X \otimes_{\sO_X} \canon_X^\dual) \otimes (\sD_X \otimes_{\sO_X} \twist_X)\bigr)[d]
        \\ \cong
        \sheafHom_{\sD_X \otimes \sD_X^\opring}\bigl(\sD_X,\, (\sD_X \otimes_{\sO_X} \twist_X) \otimes (\sD_X \otimes_{\sO_X} \twist_X)\bigr) \otimes_{\sO_X} \canon_X^\dual[d].
    \end{multline*}
    Twisting by $\canon_X$ and shifting by $d$ we obtain the desired isomorphism
    \[
        \sD_X \otimes_{\sO_X} \twist_X[d] \cong
        \sheafHom_{\sD_X \otimes \sD_X^\opring}\bigl(\sD_X,\, (\sD_X \otimes_{\sO_X} \twist_X[d]) \otimes (\sD_X \otimes_{\sO_X} \twist_X[d])\bigr).
        \qedhere.
    \]
\end{proof}

\begin{Lemma}\label{lem:hom_and_dual}
    For any $\sheaf M \in \catDbcohDMod{X}$ and $\sheaf N \in \catDbDMod{X}$ there exists a canonical isomorphism
    \begin{align*}
        \sheafHom_{\sD_X}(\sheaf M,\, \sheaf N \otimes_{\sO_X} \twist_X) & \cong
        (\canon_X \otimes_{\sO_X} \DVerdier_X\sheaf M) \otimes_{\sD_X} \sheaf N \\
        &\cong
        \canon_X \otimes_{\sD_X} (\DVerdier_X\sheaf M \otimes_{\sO_X} \sheaf N).
    \end{align*}
\end{Lemma}

\begin{proof}
    One has
    \begin{align*}
        \sheafHom_{\sD_X}(\sheaf M,\, \sheaf N \otimes_{\sO_X} \twist_X)  & \cong
        \sheafHom_{\sD_X}(\sheaf M,\, \sD_X \otimes_{\sO_X} \twist_X) \otimes_{\sD_X} \sheaf N \\&\cong
        (\canon_X \otimes_{\sO_X} \DVerdier_X\sheaf M) \otimes_{\sD_X} \sheaf N.
    \end{align*}
    The second equivalence follows from Lemma~\ref{lem:tensor_product_switching}.
\end{proof}

\subsection{Pullback}\label{sec:pullback}

Let $f\colon X \to Y$ be a morphism of smooth idealized log varieties.
While the pushforward $f_\spf$ is defined as in the classical setting, the definition of the pullback $f^!$ needs some care in order for the expected adjunctions to exist.
Recall that in the classical theory, the functor $f^!$ differs from the naive pullback $f^*$ by a shift by $\dim X - \dim Y$.
Thus, following the spirit of Remark~\ref{rem:why_twist}, in the logarithmic theory one must twist by an appropriate relative version of $\twist_X$.

If $\sheaf F$ is a (complex of) coherent $\sO_Y$-module(s), we write $f^!_{\sO\text{-mod}}$ for the $\sO$-module $!$-pullback along $f$.
Further we write $\DSerre_Y(\sheaf F) = \sheafHom_{\sO_Y}(\sheaf F,\, \Odc_{\ul Y})$ for the $\sO$-module dual of $\sheaf F$.
These two operations are related by $f^!_{\sO\text{-mod}} = \DSerre_X \circ f^* \circ \DSerre_Y$, and if $f$ is Tor-finite\footnote{Recall that a morphism $f\colon X \to Y$ of varieties is called \emph{Tor-finite} if there exists an integer $m$ such that $H^i(f^*\sheaf F) = 0$ for all $\sheaf F \in \catMod{\sO_X}$ and $i < m$.} one has $f^!_{\sO\text{-mod}}\sheaf F = f^*\sheaf F \otimes_{\sO_X} f^!_{\sO\text{-mod}}\sO_Y$.
One always has $f^!_{\sO\text{-mod}}\Odc_{\ul Y} = \Odc_{\ul X}$.

Let now $\sheaf M \in \catDbcohDModop{Y}$ have $\sO_Y$-coherent cohomology sheaves.
Then $\DSerre_Y\sheaf M$ is canonically a left $\sD_Y$-module and hence $f^!_{\sO\text{-mod}}\sheaf M = \DSerre_X \circ f^* \circ \DSerre_Y (\sheaf M)$ is canonically a right $\sD_X$-module.
In particular $\Odc_{\ul X/\ul Y} = f^!_{\sO\text{-mod}}\sO_Y$ is canonically a right $\sD_X$-module.

The right $\sD_Y$-module structure on $\sO_Y$ induces a left module structure on $\canon_Y^\dual = \sheafHom_{\sO_X}(\canon_Y,\, \sO_Y)$ and hence a right $\sD_X$-module structure on $\canon_{X/Y} = \canon_X \otimes_{\sO_X} f^*\canon_Y^\dual$.
Set
\[
    \twist_{X/Y} = \sheafHom_{\sO_X}(\canon_{X/Y},\, \Odc_{\ul X/\ul Y}) \in \catDbcohDMod{X}.
\]

\begin{Definition}\label{def:!-pullback_tor_finite}
    Let $f\colon X \to Y$ be a Tor-finite morphism of smooth idealized log varieties.
    Define a functor $f^!\colon \catDmqcDMod{Y} \to \catDmqcDMod{X}$ by
    \[
        f^!(\sheaf M) = f^*\sheaf M \otimes_{\sO_X} \twist_{X/Y}.
    \]
\end{Definition}

Recalling that for a closed immersion $i\colon Z \to X$ one has $i^!_{\sO\text{-mod}}\sheaf F = \sheafHom_{i^{-1}\sO_X}(\sO_Z,\, \sheaf F)$, a straightforward computation shows that for a Tor-finite closed immersion Definitions~\ref{def:!-pullback_tor_finite} and~\ref{def:!-pullback_closed_immersion} agree.

Any morphism $f\colon X \to Y$ can be canonically factored into a closed immersion $X \to X \times Y$ followed by projection $X \times Y \to Y$.
Having defined the $!$-pullback for each of those factors, one obtains $f^!$ in general.

\begin{Remark}
    This also provides a nice explanation for the definition of $f^!$ in the classical theory of D-modules. 
    For this one notes that every morphism of classically smooth varieties is Tor-finite.
    Hence in this case Definition~\ref{def:!-pullback_tor_finite} suffices to cover the general case.
    If the log structures on $X$ and $Y$ are trivial, then $\twist_{X/Y} \cong \sO_X[\dim X-\dim Y]$.
    In other words, even in the classical setting $f^!$ should be defined via the exceptional $\sO$-module pullback, but switching between left and right modules reduces this to a shift of the naive pullback.

    It might be interesting to see how, for general morphisms $f$ of smooth log varieties, one can obtain a $\sD_X$-module structure directly on (a suitable modification of) $f^!_{\sO\text{-mod}}\sheaf M$.
\end{Remark}

\begin{Proposition}
    Let $f\colon X \to Y$ be a Tor-finite projective morphism of smooth idealized log varieties.
    Then for any $\sheaf M \in \catDbcohDMod{X}$ and $\sheaf N \in \catDbDMod{Y}$ there exists a canonical isomorphism
    \[
        \sheafHom_{\sD_Y}(f_\spf \sheaf M,\, \sheaf N)
        \cong
        f_*\sheafHom_{\sD_X}(\sheaf M,\, f^!\sheaf N).
    \]
\end{Proposition}

\begin{proof}
    One has an isomorphism $\twist_X = \twist_{X/Y} \otimes_{\sO_X} f^*\twist_Y$.
    Set $\sheaf N' = \sheafHom_{\sO_Y}(\twist_Y, \sheaf N)$.
    Then $\sheaf N = \sheaf N' \otimes_{\sO_Y} \twist_Y$ and $f^!\sheaf N = f^!\sheaf N' \otimes_{\sO_X} f^*\twist_Y = f^*\sheaf N' \otimes_{\sO_X}\twist_X$.
    Thus by Lemma~\ref{lem:hom_and_dual} and Theorem~\ref{thm:proper_duality},
    \begin{align*}
        f_*\sheafHom_{\sD_X}(\sheaf M,\, f^!\sheaf N) & \cong
        f_*\sheafHom_{\sD_X}(\sheaf M,\, f^*\sheaf N' \otimes_{\sO_X} \twist_X) \\ & \cong
        f_*\bigl( (\canon_X \otimes_{\sO_X} \DVerdier_X\sheaf M) \otimes_{\sD_X} \sD_{X \to Y} \otimes_{f^{-1}\sD_Y}  f^{-1}\sheaf N'\bigr) \\ & \cong
        f_*\bigl( (\canon_X \otimes_{\sO_X} \DVerdier_X\sheaf M) \otimes_{\sD_X} \sD_{X \to Y} \bigr)\otimes_{\sD_Y}  \sheaf N' \\ & \cong
        (\canon_Y \otimes_{\sO_Y} f_\spf \DVerdier_X \sheaf M) \otimes_{\sD_Y} \sheaf N' \\ & \cong
        (\canon_Y \otimes_{\sO_Y} \DVerdier_X f_\spf \sheaf M) \otimes_{\sD_Y} \sheaf N' \\ & \cong
        \sheafHom_{\sD_Y}(f_\spf \sheaf M,\, \sheaf N' \otimes_{\sO_Y} \twist_Y) = 
        \sheafHom_{\sD_Y}(f_\spf \sheaf M,\, \sheaf N).
        \qedhere
    \end{align*}
\end{proof}

\begin{Remark}
    Let us note that the usual base-change theorem does not generalize to the theory of logarithmic D-modules.
    For example, if $\pt \to \as 1$ it the inclusion of the (idealized) log point into the origin, then $i^!i_\spf$ is not equal to the identity for the same reason as in the $\sO$-module case.
    Indeed,
    \[
        i^!i_\spf \CC \cong  \Hom_{\sD_{\as 1}(\as 1)}(\CC[\partial],\, \CC) \cong
        \Hom_{\sD_{\as 1}(\as 1)}\bigl(\sD_{\as 1}(\as 1) \xrightarrow{x} \sD_{\as 1}(\as 1),\, \CC\bigr) \cong \CC \oplus \CC[-1].
    \]
    One would expect this issue to be resolved by a suitable theory of \enquote{derived logarithmic geometry.}
\end{Remark}

\section{Pushforward of holonomic modules}\label{sec:pushforward_hol}

Recall that a $\sD_X$-module $\sheaf M$ is called holonomic if $\logdim \Ch(\sheaf F) = \logdim X$.
If the characteristic variety of $\sheaf M$ is contained in the zero section of $T^*X$, then $\sheaf M$ is necessarily $\sO_X$-coherent.
In this case we call $\sheaf M$ an integrable (log) connection.
As in the classical setting if $\sheaf M$ is holonomic there always exists a divisor $Z$ on $X$ such that $\res{\sheaf M}{X\setminus Z}$ is an integrable connection.

The aim of this section is to prove the following fundamental theorem.

\begin{Theorem}\label{thm:projective_pf_preserves_holonomic}
    Let $f\colon X \to Y$ be a projective morphism of smooth idealized log varieties and assume that the morphism on characteristics $f^\flat\colon f^{-1}\logMbar_Y \to \logMbar_X$ is injective.
    Assume that $\sheaf M \in \catDbholDMod{X}$ admits locally on $Y$ a good filtration.
    Then $f_\spf\sheaf M$ is holonomic.
\end{Theorem}

\subsection{Kashiwara's Estimate}

If $f\colon X \to Y$ is a morphism of smooth idealized log varieties then we have an associated diagram of log cotangent bundles
\[
    \begin{tikzcd}
        T^*X &
        \arrow[l, "\rho_f"'] f^*(T^*Y) \arrow[r, "\varpi_f"] &
        T^*Y.
    \end{tikzcd}
\]

\begin{Proposition}[Kashiwara's estimate for the characteristic variety]\label{prop:Kashiwara_estimate}
    Let $f\colon X \to Y$ be a proper morphism of smooth log varieties and assume that $\sheaf M \in \catCohDMod{X}$ has locally on $Y$ a good filtration.
    Then
    \[
        \Ch(f_{\spf}\sheaf M) \subseteq \varpi_f(\rho_f^{-1}\Ch \sheaf M).
    \]
\end{Proposition}

As we currently do not have a theory of logarithmic microdifferential operators, we will follow the proof using the Rees construction given in \cite{MaisonobeSabbah:Kaiserslautern} (which is in turn based on the ideas in \cite{Malgrange:1985:ImagesDirectes}).
Since to the author's knowledge this proof has not appeared in peer-reviewed literature we will reproduce it here.
Of course the author does not claim any originality.

Let $(\sheaf A,\, F_\bullet)$ be a filtered sheaf of $\CC$-algebras.
Recall that the Rees sheaf of rings $R(\sheaf A)$ is the subsheaf $\bigoplus_n F_n\sheaf A z^n$ of $\sheaf A \otimes_{\CC} \CC[z,z^{-1}]$.
If $(\sheaf M,\, F_\bullet)$ is a filtered $\sheaf A$-module, then $R(\sheaf M)$ is similarly defined to be the $R(\sheaf A)$-module $\bigoplus_n F_n\sheaf M z^n$.
It follows that
\[
    R(\sheaf M)/zR(\sheaf M) = \gr \sheaf M
    \quad\text{and}\quad
    R(\sheaf M)/(z-1)R(\sheaf M) = \sheaf M.
\]
Any graded $R(\sheaf A)$-module without $\CC[z]$-torsion is canonically obtained from a filtered $\sheaf A$-module.

We will apply this to $\sD_X$ with its usual filtration.
If $\sheaf M$ is a  coherent $\sD_X$-module with a global good filtration, then the associated Rees module $R(\sheaf M)$ is a coherent $R(\sD_X)$-module.

For a morphism $f\colon X \to Y$ of smooth idealized log varieties we define a functor $f^R_{\spf}\catDp{R(\sD_X)} \to \catDp{R(\sD_Y)}$ by $f^R_{\spf}\sheaf N = f_*\bigl(R(\sD_{Y \from X}) \otimes_{R(\sD_X)} \sheaf N\bigl)$.

\begin{Lemma}\label{lem:Rees_pf_properties}
    Let $f\colon X \to Y$ be a morphism of smooth log varieties.
    \begin{enumerate}
        \item\label{it:lem:Rees_pf_properties:1} If $f$ is proper, then $f^R_{\spf}$ preserves coherence.
        \item\label{it:lem:Rees_pf_properties:2} For any $\sheaf N \in \catMod{R(\sD_X)}$ and any $i$ there exists a canonical equivalence
            \[
                H^i(f^R_{\spf}\sheaf N) / (z-1)H^i(f^R_{\spf}\sheaf N)
                \cong
                H^i\bigl(f_{\spf}(\sheaf N/(z-1)\sheaf N)\bigr).
            \]
        \item\label{it:lem:Rees_pf_properties:3} Assume that $\sheaf M \in \catCohDMod{X}$ has a global good filtration.
            Then the quotient of $H^i(f^R_{\spf} R(\sheaf M))$ by its $z$-torsion is the Rees module associated to some good filtration on $H^i(f_{\spf}\sheaf M)$.
    \end{enumerate}
\end{Lemma}

\begin{proof}
    Assertion \ref{it:lem:Rees_pf_properties:1} can be shown in the same way as the corresponding statement for $\sD_X$-modules (Lemma~\ref{lem:proper_pushforward_preserves_coherence}).

    For \ref{it:lem:Rees_pf_properties:2} we first note that for any complex of modules $N$ over a Rees ring one has $H^i(N/(z-1)N) \cong H^i(N)/(z-1)H^i(N)$.
    The tensor product $R(\sD_{Y \from X}) \otimes_{R(\sD_X)} \sheaf N$ in the definition of $f^R_{\spf}$ can be computed with a Spencer-type resolution of $R(\sD_{Y \from X})$ with $\Sp(R(\sD_{Y \from X}))/(z-1)\Sp(R(\sD_{Y \from X})) \cong \Sp(\sD_{Y \from X})$.
    Finally, the derived functor $f_*$ can be computed via a Godement resolution of its argument.
    Restricting the Godement resolution of a sheaf $\sheaf F$ over the Rees algebra to $z=1$ reproduces the Godement resolution for $\sheaf F/(z-1)\sheaf F$.

    To show \ref{it:lem:Rees_pf_properties:3} let $T$ be the $z$-torsion submodule of $H^i(f^R_\spf(R(\sheaf M)))$.
    The module $H^i(f^R_\spf(R(\sheaf M))) / T$ is $z$-torsion free and hence is the Rees module of some good filtration on
    \begin{align*}
        \rquot{H^i(f^R_\spf(R(\sheaf M))) / T}{(z-1)(H^i(f^R_\spf(R(\sheaf M))) / T)} &\cong
        \rquot{H^i(f^R_\spf(R(\sheaf M)))}{(z-1)H^i(f^R_\spf(R(\sheaf M)))} \\ &\cong
        H^i\bigl(f_\spf\bigl(R(\sheaf M)/(z-1)R(\sheaf M)\bigr)\bigr) \\ &\cong
        H^i(f_\spf(\sheaf M)).
        \qedhere
    \end{align*}
\end{proof}

\begin{proof}[Proof of Proposition~\ref{prop:Kashiwara_estimate}]
    As the statement is local on $Y$, one can assume that $\sheaf M$ has a good filtration $F_\bullet$.

    Define a functor $f^{\gr}_\spf\colon \catMod{\gr \sD_X} \to \catMod{\gr \sD_Y}$ by $f^{\gr}_\spf = \varpi_{f,*} \circ \rho_f^*$.
    One sees that
    \[
        \supp H^i (f^{\gr}_\spf \gr \sheaf M) \subseteq \varpi_f(\rho_f^{-1}\Ch\sheaf M).
    \]
    The remainder of the proof consists of comparing the support of $\gr (f_\spf \sheaf M)$ with that of $f^{\gr}_\spf(\gr \sheaf M)$.
    We will do so by interpolating with the help of the Rees construction.

    As $H^i(f^R_\spf(R(\sheaf M)))$ is coherent, the sequence of submodules given by the kernel of multiplication by $z^j$ (locally) stabilizes.
    Thus we can assume that the $z$-torsion submodule of $H^i(f^R_\spf(R(\sheaf M)))$ is given by $\ker z^\ell$ for some fixed $\ell$.
    
    We will denote by $\gr_{[\ell]}\sheaf M = \bigoplus_k F_k\sheaf M/F_{k-\ell}\sheaf M$ the grading with step $\ell$.
    One defines the functor $f_\spf$ for $\gr_{[\ell]}\sD_X$-modules in the same manner as for $\sD_X$-modules.

    It follows from Lemma~\ref{lem:Rees_pf_properties}\ref{it:lem:Rees_pf_properties:3} that $\gr_{[\ell]} H^i(f_\spf \sheaf M) = \rquot{R(H^i(f_\spf\sheaf M))}{z^\ell R(H^i(f_\spf \sheaf M))}$ is a quotient of $\rquot{H^i(f^R_\spf R(\sheaf M))}{z^\ell H^i(f^R_\spf R(\sheaf M))}$.
    The exact sequence
    \[
        \cdots
        \to
        H^i(f^R_\spf R(\sheaf M))
        \xrightarrow{z^\ell}
        H^i(f^R_\spf R(\sheaf M))
        \to
        H^i(f_\spf(R(\sheaf M)/z^\ell R(\sheaf M)))
        \to
        \cdots
    \]
    implies that in turn $\rquot{H^i(f^R_\spf R(\sheaf M))}{z^\ell H^i(f^R_\spf R(\sheaf M))}$ is a submodule of $H^i(f_\spf \gr_{[\ell]}\sheaf M) = H^i(f_\spf(R(\sheaf M)/z^\ell R(\sheaf M)))$.
    Thus $\gr_{[\ell]} H^i(f_\spf \sheaf M)$ is a $\gr_{[\ell]}\sD_Y$-subquotient of $H^i(f_\spf \gr_{[\ell]}\sheaf M)$.

    We filter $\gr_{[\ell]}\sheaf \sD_X$ by the finite filtration
    \[
        G_j \gr_{[\ell]}\sheaf \sD_X = \bigoplus_k F_{k+j-\ell}\sD_X / F_{k-\ell} \sD_X, \quad 0 \le j \le \ell.
    \]
    Similarly, we have finite $G$-filtrations on $\gr_{[\ell]}\sD_X$-modules.
    One notes that $\gr^G \gr^F_{[\ell]} \cong \gr^F \sD_X[u]/u^\ell$ with $u$ in $G$-degree $1$.
    Thus given a coherent $\gr_{[\ell]}\sD_X$-module, its graded module with respect to any $G$-filtration is $\gr \sD_X[u]/u^\ell$-coherent and hence also $\gr \sD_X$-coherent.
    Further, its support does not depend on the choice of $G$-filtration.

    The $G$-filtration induces a finite spectral sequence with $E_2$-term
    \[
        H^i(f^{\gr}_\spf \gr^G\gr^F_{[\ell]} \sheaf M) = H^i(f^{\gr}_\spf \gr^F \sheaf M[u]/(u^\ell)) \cong H^i(f^{\gr}_\spf \sheaf M)^\ell
    \]
    converging to $\gr^G H^i(f_\spf \gr^F_{[\ell]} \sheaf M)$ for some suitable $G$-filtration.
    Hence the support of $\gr^G H^j(f_\spf \gr_{[\ell]}^F \sheaf M)$ is contained in $\supp H^i(f_\spf^{\gr} \sheaf M) \subseteq \varpi_f(\rho_f^{-1}\Ch\sheaf M)$.

    The $G$-filtration on $H^i(f_\spf \gr_{[\ell]} \sheaf M)$ induces a $G$-filtration on the subquotient $\gr_{[\ell]} H^i(f_\spf \sheaf M)$.
    Thus the support of $\gr^G \gr^F_{[\ell]} H^i(f_\spf \sheaf M) \cong \bigl(\gr^F H^i(f_\spf \sheaf M))^\ell$ is contained in the support of $\gr^G H^j(f_\spf \gr_{[\ell]}^F \sheaf M)$ and hence in $\varpi_f(\rho_f^{-1}\Ch\sheaf M)$.
    Therefore, so is the characteristic variety of $H^i(f_\spf \sheaf M)$.
\end{proof}

\subsection{Some reduction steps}

For classical D-modules, Theorem~\ref{thm:projective_pf_preserves_holonomic} follows from Proposition~\ref{prop:Kashiwara_estimate} together with a result from symplectic geometry.
As the logarithmic cotangent bundle does not have a natural symplectic structure, we need some further reductions steps to deal with the logarithmic situation.
The idea is to reduce to a situation with no interesting log structure.

\begin{Lemma}\label{lem:pullback_to_stratum}
    Let $i\colon Z \hookrightarrow X$ be the inclusion of (a component of) a closed log stratum, endowed with the induced idealized log structure.
    Let $\sheaf M = \sD_X \otimes_{\sO_X} \sheaf G$ for some $\sheaf G \in \catCohOMod{X}$ with $\sD_X$-module structure given by left multiplication.
    Then,
    \[
        i^!\sheaf M \cong \canon_Z^\dual \otimes_{\sO_Z} \sD_Z \otimes_{\sO_Z} i^!_{\sO\text{-mod}}(\canon_X \otimes_{\sO_X} \sheaf G).
    \]
\end{Lemma}

\begin{proof}
    In this situation one has $i^*\sD_X = \sD_Z$ and $i^*\canon_X = \canon_Z$.
    Thus,
    \begin{align*}
        i^!\sheaf M & =
        \sheafHom_{i^{-1}\sD_X}(\sD_{X \from Z},\, i^{-1}(\sD_X \otimes_{\sO_X} \sheaf G)) \\ & \cong
        \sheafHom_{i^{-1}\sO_X}(i^{-1}\canon_X^\dual \otimes_{i^{-1}\sO_X} \canon_Z,\, i^{-1}\sD_X \otimes_{i^{-1}\sO_X} i^{-1}\sheaf G) \\& \cong
        \sheafHom_{i^{-1}\sO_X}(i^*\canon_X,\, i^{-1}\canon_X \otimes_{i^{-1}\sO_X} i^{-1}\sD_X \otimes_{i^{-1}\sO_X} i^{-1}\sheaf G) \\ & \cong
        \sheafHom_{i^{-1}\sO_X}(\sO_Z \otimes_{i^{-1}\sO_X} i^{-1}\canon_X,\, i^{-1}\canon_X \otimes_{i^{-1}\sO_X} i^{-1}\sD_X \otimes_{i^{-1}\sO_X} i^{-1}\sheaf G) \\ & \cong
        i^{-1}\canon_X^\dual \otimes_{i^{-1}\sO_X} i^{-1}\sD_X \otimes_{i^{-1}\sO_X} \sheafHom_{i^{-1}\sO_X}(\sO_Z,\, i^{-1}\canon_X \otimes_{i^{-1}\sO_X} i^{-1}\sheaf G) \\ & \cong
        \canon_Z^\dual \otimes_{\sO_Z} \sD_Z \otimes_{\sO_Z} i^!_{\sO\text{-mod}}(\canon_X \otimes_{\sO_X} \sheaf G).
        \qedhere
    \end{align*}
\end{proof}

\begin{Lemma}\label{lem:projection_stratum_basechange}
    Let $X$ and $Y$ be smooth idealized log varieties with $X$ proper.
    Let $f\colon X \times Y \to Y$ be the projection and let $i\colon Z \hookrightarrow Y$ be the inclusion of (a component of) a closed log stratum (endowed with the induced idealized log structure).
    Consider the cartesian diagram
    \[
        \begin{tikzcd}
            X \times Z \arrow[r,"\tilde \imath"] \arrow[d, "\tilde f"] & X \times Y \arrow[d, "f"] \\
            Z \arrow[r, "i"] & Y
        \end{tikzcd}
    \]
    Assume that $\sheaf M \in \catDbcohDMod{X \times Y}$ admits locally on $Y$ a good filtration. 
    Then there exists a canonical isomorphism
    \[
        \tilde f_\spf \tilde\imath^! \sheaf M \isoto i^! f_\spf \sheaf M.
    \]
\end{Lemma}

\begin{proof}
    By adjunction we get a morphism
    \begin{equation}\label{eq:lem:projection_stratum_basechange:1}
        \tilde f_\spf \tilde\imath^! \to
        i^! i_\spf \tilde f_\spf \tilde\imath^! =
        i^! f_\spf \tilde \imath_\spf \tilde\imath^! \to
        i^! f_\spf.
    \end{equation}
    The question whether this is an isomorphism is local.
    Hence by the Way-Out Lemma it suffices to show that it is an isomorphism when applied to $\sheaf M = \sD_{X\times Y} \otimes_{\sO_X} \sheaf G$ for some $\sO$-coherent sheaf $\sheaf G$.
    Then, using Lemma~\ref{lem:pullback_to_stratum} and base change for coherent $\sO$-modules, one computes
    \begin{align*}
        \tilde f_\spf \tilde\imath^!(\sD_{X \times Y} \otimes_{\sO_X} \sheaf G) & \cong
        \sD_Z \otimes_{\sO_Z} \canon_Z^\dual \otimes_{\sO_Z} \tilde f_*\bigl( \canon_{X \times Z} \otimes_{\sD_{X \times Y}} i^!(\sD_{X \times Y} \otimes_{\sO_{X \times Y}} \sheaf G )\bigr)\\ &\cong
        \sD_Z \otimes_{\sO_Z} \canon_Z^\dual \otimes_{\sO_Z} \tilde f_*\bigl( \tilde\imath_{\Omod}^!(\canon_{X \times Y} \otimes_{\sO_{X \times Y}} \sheaf G) \bigr)\\ &\cong
        \sD_Z \otimes_{\sO_Z} \canon_Z^\dual \otimes_{\sO_Z} i_{\Omod}^!\bigl( f_*(\canon_{X \times Y} \otimes_{\sO_{X \times Y}} \sheaf G) \bigr)\\ &\cong
        i^!f_\spf(\sD_{X \times Y} \otimes_{\sO_{X \times Y}} \sheaf G).
        \qedhere
    \end{align*}
\end{proof}

\begin{Lemma}\label{lem:i^!_to_stratum_and_hol}
    Let $i\colon Z \hookrightarrow X$ be the inclusion of (a component of) a closed log stratum, endowed with the induced idealized log structure. 
    Then for any $\sheaf M \in \catDbcohDMod{X}$ the restriction $i^!\sheaf M$ is holonomic if and only if $\logdim \res{\Ch(\sheaf M)}{Z} = \logdim X$.
\end{Lemma}

\begin{proof}
    One notes that in this situation $\sD_{X \from Z} = \sD_Z = \sD_X \otimes_{i^{-1}\sO_X} \sO_Z$.
    Hence,
    \begin{align*}
        i_*i^! \DVerdier_X\sheaf F & \cong
        i_*\sheafHom_{i^{-1}\sD_X}\bigl(\sD_{X \from Z},\, i^{-1}(\sheafHom_{\sD_X}(\sheaf M, \sD_X \otimes_{\sO_X} \twist_X) \otimes_{\sO_X} \canon_X^\dual)\bigr) \\ & \cong
        i_*\sheafHom_{i^{-1}\sO_X}\bigl(\sO_Z,\, i^{-1}(\sheafHom_{\sD_X}(\sheaf M, \sD_X \otimes_{\sO_X} \twist_X) \otimes_{\sO_X} \canon_X^\dual)\bigr) \\ & \cong
        \sheafHom_{\sO_X}\bigl(\sO_Z,\, \sheafHom_{\sD_X}(\sheaf M, \sD_X \otimes_{\sO_X} \twist_X)\bigr) \otimes_{\sO_X} \canon_X^\dual \\ & \cong
        \sheafHom_{\sD_X}\bigl(\sO_Z \otimes_{\sO_X} \sheaf M,\, \sD_X \otimes_{\sO_X} \twist_X\bigr) \otimes_{\sO_X} \canon_X^\dual \\ & \cong
        \DVerdier_X(\sO_Z \otimes_{\sO_X} \sheaf M).
    \end{align*}
    Thus the result follows from the fact that duality preserves holonomicity (\cite[Proposition~3.24]{KT}) and \cite[Lemma~3.17]{KT}.
\end{proof}

Assume that the rank of $\logMbargp_X$ is constant, i.e.~$X$ consists of a single log stratum.
We note that if the rank is positive, then such $X$ is always idealized.
Then $\ul X$ is necessarily smooth and by the description of $\sD_X$ given in Proposition~\ref{prop:D_X_in_coords} $n=\dim \ul X$, so that $\sD_{\ul X}$ naturally sits inside of $\sD_X$.

\begin{Lemma}\label{lem:single_stratum_hol}
    Let $X$ be a smooth idealized log variety and assume that the rank of $\logMbargp_X$ is constant.
    Then $\sheaf M \in \catCohDMod{X}$ is holonomic if and only if it is coherent and holonomic as a $\sD_{\ul X}$-module.
\end{Lemma}

\begin{proof}
    The question is local, so that we can assume that $\sheaf M$ has a good filtration and $\sD_X \cong \sD_{\ul X}[\partial_{m_1},\dotsc,\partial_{m_k}]$, where the generators $\partial_{m_i}$ commute with everything.

    Let $Z$ be a closed subvariety of $X$ defined by an ideal sheaf $\sheaf J$.
    We set $\sheaf M(*Z) = \varinjlim_m \sheafHom_{\sO_X}(\sheaf J^m, \sheaf M)$.
    If $j \colon U = X \setminus Z \hookrightarrow X$ is the inclusion, then $\sheaf M(*Z)$ is a subsheaf of $j_*j^*\sheaf M$ (indeed in the algebraic setting these two sheaves are the same).
    In particular, if $\sheaf M' \subseteq \sheaf M$ and $\res{\sheaf M}{U} = \res{\sheaf M'}{U}$ then $(\sheaf M/\sheaf M')(*Z) = 0$ and hence $\sheaf M(*Z) \cong \sheaf M'(*Z)$.
    Further there is a canonical map $\sheaf M \to \sheaf M(*Z)$ whose kernel is supported on $Z$.

    Since in the present situation $\sD_X$ differs from $\sD_{\ul X}$ only by the addition of some commuting operators which act trivially on $\sO_X$, many results about $\sD_{\ul X}$-modules also apply to $\sD_X$-modules with virtually unchanged proofs.
    In particular $\sheaf M(*Z)$ is again a $\sD_X$-module \cite[Lemma~1.1]{Kashiwara:1978:HolII}, and if $\res{\sheaf M}{U}$ is holonomic, the so is $\sheaf M(*Z)$ by the second main theorem of \cite{Kashiwara:1978:HolII}.
    In fact more is true: If $\res{\sheaf M}{U}$ is coherent and holonomic as a $\sD_{\ul X}$-module, then so is $\sheaf M(*Z)$.
    Indeed, since we have a global good filtration, there exists a $\sD_{\ul X}$-coherent submodule which coincides with $\sheaf M$ on $U$.
    By the above observation, replacing $\sheaf M$ by this submodule does not change $\sheaf M(*Z)$.

    We further note that in the present situation Kashiwara's Equivalence holds (with the usual proof), that is, if $Z$ is smooth and $i$ the inclusion then $i_\spf$ is exact and induces an equivalence of the category $\catQCDMod{Z}$ with the subcategory of $\catQCDMod{X}$ consisting of modules supported on $Z$.
    This equivalence respects coherence and holonomicity, both as $\sD_X$ and $\sD_{\ul X}$-modules.

    We are now ready to prove the \enquote{only if} statement of the lemma by induction on the dimension of the support of $\sheaf M$.
    If $\dim \supp \sheaf M = 0$ then the result is just an application of Kashiwara's equivalence.

    Set $Z = \supp \sheaf M$ and let $Z_0$ be the singular locus of $Z$.
    Assume that we have shown that $\res{\sheaf M}{X \setminus Z_0}$ is $\sD_{\ul Z \setminus \ul Z_0}$-coherent and holonomic.
    Then $\sheaf M(*Z_0)$ is $\sD_{\ul X}$-coherent and holonomic, and the cone of the morphism $\sheaf M \to \sheaf M(*Z_0)$ is $\sD_X$-coherent, holonomic and supported on $Z_0$.
    We are thus done by induction.

    Hence we can assume that $Z$ is nonsingular and apply Kashiwara's equivalence to replace $X$ by $Z$.
    Then there exists a divisor $D$ such that $\res{\sheaf M}{X \setminus D}$ is $\sO_X$-coherent and thus certainly $\sD_{\ul X}$-coherent and holonomic.
    Taking the cone of the map $\sheaf M \to \sheaf M(*D)$ we reduce to a module supported on $D$ and apply induction.

    The proof of the other implication proceeds in a similar manner.
\end{proof}

\subsection{Proof of Theorem~\ref{thm:projective_pf_preserves_holonomic}}

We factor $f$ into a closed immersion $X \hookrightarrow \ps n \times Y$ followed by a projection $\ps n \times Y \to Y$, where we endow $\ps n$ with the trivial log structure.

If $i \colon Z \to X$ is any closed immersion such that $i^{-1}\logMbar_X \to \logMbar_Z$ is injective, then the fibers of $\rho\colon \res{T^*X}{Z} \to T^*Z$ have at most dimension $\dim X - \dim Z$ (they may be empty).
Hence Proposition~\ref{prop:Kashiwara_estimate} shows that $i_{\spf}$ preserves holonomicity.

Let new $p\colon \ps n \times Y \to Y$ be the projection.
By Lemma~\ref{lem:i^!_to_stratum_and_hol} it suffices to check that the $!$-pullback of $p_{\spf}\sheaf M$ to each log stratum is holonomic.
The statement is local on $Y$, so that using Lemma~\ref{lem:projection_stratum_basechange}, we can thus assume that $Y$ consists of a single log stratum.
But then holonomicity as $\sD_Y$-modules coincides with holonomicity as $\sD_{\ul Y}$-modules by Lemma~\ref{lem:single_stratum_hol}, so that the statement is classical.
\qed

\section{Sheaves on the Kato--Nakayama space}

Let $X$ be an analytic log variety.
Following \cite{KatoNakayama} the logarithmic de Rham functor should take values in sheaves on the so-called Kato--Nakayama space $\KN X$.
Recall that as a set $\KN X$ is defined to be all pairs $(x, \sigma)$, with $x \in X$ and $\sigma\colon \logM_{X,x} \to S^1$ a homomorphism of monoids fitting into the commutative diagram
\[
    \begin{tikzcd}
        \logM_{X,x} \arrow[r, "\sigma"] & S^1 \\
        \sO_{X,x}^* \arrow[u] \arrow[r, "x^*"] & \CC^* \arrow[u, "\operatorname{arg}"].
    \end{tikzcd}
\]
We refer to \cite[Section~1]{KatoNakayama} and \cite[Section~V.1.2]{logbook} for details.
If $X$ is a smooth (idealized) log variety, $\KN X$ has the structure of a topological manifold with boundary whose interior is identified with $X^* = X \setminus X^1$ (resp.~a topological $(S^1)^e$-bundle over $X \setminus X^1$ in the idealized case, where $e$ is the generic rank of $\logMbargp$).
We let $\tau_X \colon \KN X \to X$ be the projection.
It is proper with and its fibers $\tau_X^{-1}(x)$ have a structure of a torsor under the real torus $\Hom(\logMbargp_{X,x},\, S^1) \cong (S^1)^{\mathrm{rk} \logMbargp_{X,x}}$.

If $\ul X$ is classically smooth and the log structure on $X$ is given by a simple normal crossings divisor $D$, then $\KN X$ is the real blowup of $\ul X$ along $D$.

\subsection{Graded topological spaces}

The de Rham functor will send D-modules to certain graded sheaves on the Kato--Nakayama space.
A theory of such graded sheaves was developed in \cite{Koppensteiner:graded_top}.
In order to fix notation, we give a brief overview of the basic definitions of this theory.

Thus in this subsection $X$ will denote be a topological space, which we will always assume to be locally compact and Hausdorff.
We fix a commutative noetherian base ring $k$.
A \emph{graded topological space} $(X, \Lambda)$ consists of a topological space $X$ together with a sheaf of abelian groups $\Lambda$.
A morphism of graded topological spaces $(X,\Lambda_X) \to (Y,\Lambda_Y)$ consists of a pair $(f,f^\flat)$, where $f\colon X \to Y$ is a continuous map and $f^\flat\colon f^{-1}\Lambda_Y \to \Lambda_X$ is a morphism of sheaves of abelian groups.

A \emph{($\Lambda$-graded) presheaf} $\sheaf F$ on $(X, \Lambda)$ assigns to each open subset $U$ of $X$ a $\Lambda(U)$-graded $k$-module $\sheaf F(U)$ together with a restriction map $\rho^U_V$ for each open inclusion $V \subseteq U$ such that $\rho^U_V(\sheaf F(U)_\lambda) \subseteq \sheaf F(V)_{\res{\lambda}{V}}$ for all $\lambda \in \Lambda(U)$.
For any $\lambda \in \Lambda(X)$ we write $\sheaf F_\lambda$ for the ordinary presheaf $U \mapsto \sheaf F(U)_{\res{\lambda}{U}}$ and $\sheaf F\langle \lambda \rangle$ for the graded presheaf with $\sheaf F\langle \lambda \rangle(U)_\mu = \sheaf F(U)_{\res{\lambda}{U} + \mu}$.
A $\Lambda$-graded presheaf $\sheaf F$ is called a \emph{($\Lambda$-graded) sheaf} if for each open subset $U$ of $X$ and each $\lambda \in \Lambda(U)$ the ordinary presheaf $(\res{\sheaf F}{U})_\lambda$ is a sheaf.
Morphism of graded (pre-) sheaves are defined in the obvious way by the requirement to be preserve the grading on each open subset.

Let $f\colon (X,\Lambda_X) \to (Y,\Lambda_Y)$ be a morphism of graded topological spaces.
One defines an adjoint pair of functors between the corresponding categories of graded sheaves as follows.
The functor
\[
    f_{\gr}^{-1}\colon \catSh{Y,\Lambda_Y} \to \catSh{X,\Lambda_X}
\]
is given by
\[
    \Gamma(U,f_{\gr}^{-1}\sheaf F)_\lambda = \bigl\langle s \in \Gamma(U,f^{-1}\sheaf F) : f^\flat(\deg s) = \lambda\bigr\rangle, \quad \lambda \in \Lambda_X(U).
\]
Its right adjoint
\[
    f_{\gr,*}\colon \catSh{X,\Lambda_X} \to \catSh{Y,\Lambda_Y}
\]
is defined as
\[
    \Gamma(V,f_{\gr,*}\sheaf F)_\mu = \Gamma(f^{-1}V,\, \sheaf F)_{f^\flat(\mu)},\quad \mu \in \Lambda_Y(V) \to f^{-1}\Lambda_Y(f^{-1}V).
\]
This theory can be further enhanced by adding a sheaf of rings so that a \emph{ringed graded space} is a triple $(X,\Lambda, \sheaf R)$ such that $(X,\Lambda)$ is a graded topological space and $\sheaf R$ is a $\Lambda$-graded sheaf of $k$-algebras.
A morphism of graded ringed topological spaces $(X,\Lambda_X,\sheaf R_X) \to (Y,\Lambda_Y,\sheaf R_Y)$ is a triple $(f,f^\flat,f^\sharp)$ where $(f,f^\flat)$ is a morphism of graded topological spaces and $f^\sharp\colon f_{\gr}^{-1}\sheaf R_Y \to \sheaf R_X$ is morphism of $\Lambda_X$-graded sheaves of $k$-algebras.
Such a morphism is called \emph{strict} if $f^\flat$ and $f^\sharp$ are isomorphisms.
($\Lambda$-graded) sheaves of $\sheaf R$-modules are defined in the obvious way.

The functor $f_{\gr,*}$ immediately upgrades to a functor of graded sheaves of modules.
Its left adjoint is
\[
    f_{\gr}^*\colon \catSh{Y,\Lambda_Y,\sheaf R_Y} \to \catSh{Y,\Lambda_X,\sheaf R_X},
    \qquad
    \sheaf F \mapsto f_{\gr}^{-1}\sheaf F \otimes_{f_{\gr}^{-1}\sheaf R_Y} \sheaf R_X.
\]
Besides $f_{\gr,*}$ one also has the direct image with proper supports $f_{\gr,!}\sheaf F$, which is the subsheaf of $f_{\gr,*}\sheaf F$ with sections
\[
    \Gamma(U, f_{\gr,!}\sheaf F)_\mu = \bigl\{ s \in \Gamma(f^{-1}U, \sheaf F)_{f^\flat(\mu)} : f\colon \supp s \to U \text{ is proper}\bigr\}.
\]
On the main result of \cite{Koppensteiner:graded_top} is an extension of Poincar\'e--Verdier duality to ringed grades spaces, stating that on the level of derived categories $f_{\gr,!}$ has a right adjoint $f_{\gr}^!$.
Let $p\colon (X, \Lambda, \sheaf R) \to (\pt, 0, k)$ be the canonical morphism.
When $k$ is a field one calls $p_{\gr}^!k$ the \emph{dualizing complex} of $(X,\Lambda,\sheaf R)$.

\subsection{The ring \texorpdfstring{$\Clog_X$}{Clog} and modules over it}\label{sec:Clog}

To classify arbitrary log connections (or $\sD_X$-modules) it is not sufficient to consider sheaves of $\CC$-modules on $\KN X$.
For this reason Ogus \cite{Ogus} introduced a graded sheaf of rings on $\KN X$ whose definition we will recall now.
Set $\Lambda = \CC \otimes \logMbargp_X$.
By abuse of notation we will also write $\Lambda$ for the pullback $\tau_X^{-1}\Lambda$.
Let $\logK_X$ be the sheaf of ideals defining the idealized structure of $X$.
(The reader not interested in the idealized setting may assume $\logK_X$ to be zero.)
The sheaf $\Clog_X$ is defined to be the pullback by $\tau_X$ of the quotient of the sheaf of monoid algebras $\CC[-\logMbar_X]$ by the ideal generated by $-\logK_X$.

The inclusion $\logMbar_X \to \Lambda$ gives $\Clog_X$ the structure of a $\Lambda$-graded sheaf of rings.
We write $\catClogMod{X}$ for the category of $\Lambda$-graded sheaves of $\Clog_X$-modules and $\catDbClogMod{X}$ for its bounded derived category.
Ogus further defines the subcategory $\catLogLocSys{X}$ of \emph{coherent} $\Clog_X$-modules \cite[Definition~3.2.4]{Ogus}, which are the correct analogue of local systems in this setting.
He then proves that this category is equivalent to the category of logarithmic connections on $X$ \cite[Theorem~3.4.2]{Ogus}, where a logarithmic connection is just an $\sO_X$-coherent $\sD_X$-module.

A morphism $f\colon X \to Y$ of analytic log varieties induces a morphism $f_{\log}\colon \KN X \to \KN Y$ of ringed graded topological spaces fitting into the commutative diagram
\[
    \begin{tikzcd}
        \KN{X} \arrow[r, "f_{\log}"] \arrow[d, "\tau_X"] & \KN Y \arrow[d, "\tau_Y"] \\
        X \arrow[r, "f"]{Y} & Y.
    \end{tikzcd}
\]
If $f$ is strict then this diagram is cartesian.

The image of the logarithmic de Rham functor will not be arbitrary sheaves of $\Clog_X$-modules.
As in \cite{Ogus:2003:LogarithmicRiemannHilbertCorrespondence}, we will impose a monodromy condition along the fibers of $\tau_X$.
For this we note that for any point $x \in X$ the logarithmic inertia group $I_x = \Hom(\logMbargp_{X,x},\, 2\pi i \ZZ)$ is canonically identified with the fundamental group of $\tau_X^{-1}(x)$.
For any $x \in X$ there exists a natural pairing $\langle {-},{-} \rangle \colon I_{x} \otimes \Lambda_x \to \CC$.
If $m \in \logMbargp_{X,x}$, then $\exp\langle {-},m\rangle = 1$.

We write $\catClogModTau{X}$ for the full subcategory of $\catClogMod{X}$ consisting of sheaves which are locally constant along fibers of $\tau_X$.
The inclusion $\catClogModTau{X} \to \catClogMod{X}$ has a right adjoint $\Psi^\tau$ given by associating to any sheaf $\sheaf F$ the largest subsheaf of $\sheaf F$ contained in $\catClogModTau{X}$.

For any $\sheaf F \in \catClogModTau{X}$ and any point $z$ of $\KN X$ there is a natural action $\rho$ of $I_{\tau(z)}$ on the stalk of $\sheaf F$ at $z$.
Using this action we let $\catClogModMon{X}$ be the full subcategory of $\catClogModTau{X}$ consisting of sheaves $\sheaf F$ such that for each $z \in \KN X$, $\gamma \in I_{\tau_X(z)}$ and $\lambda \in \Lambda_z$ the action of $\rho_\gamma - \exp\langle \gamma, \lambda \rangle$ on $\sheaf F_{\lambda,z}$ is locally nilpotent, that is for any $s \in \sheaf F_{\lambda,z}$ there exists a positive integer $N$ such that $(\rho_\gamma - \exp\langle \gamma, \lambda \rangle)^N s = 0$.
The inclusion $\catClogModMon{X} \to \catClogModTau{X}$ has a right adjoint $\Psi^\Lambda$ which is explicitly given by associating to a sheaf $\sheaf F$ the subsheaf $\Psi^\Lambda(\sheaf F)$ whose sections over $U$ consist of all those sections $s \in \sheaf F(U)_\lambda$ such that for each point $z \in U$ the action of $\rho_\gamma - \exp\langle \gamma, \lambda \rangle$ on $s_z$ is nilpotent.
Set
\[
    \Psi = \Psi^\Lambda \circ \Psi^\tau\colon \catClogMod{X} \to \catClogModMon{X}.
\]

Let $p\colon \KN X \to \pt$ be the natural map to the (ungraded) point.
The dualizing complex of $\catDbClogMod{X}$ is $\logdc_{\KN X} = p_{\gr}^!\CC$.
In order to describe this complex, we need to introduce some additional notation:
Let $X^k_j$ be the irreducible components of $X^k$.
Each $X^k_j$ is a smooth idealized log space and the inclusion $X^k_j \to X$ induces a map on the corresponding Kato--Nakayama spaces, and in particular a morphism $\Clog_X \to \Clog_{X^k_j}$.
Set 
\[
    \sheaf D^i = \bigoplus_j \Clog_{X^{i + \dim X}_j}, \qquad -\dim X \le i \le 0
\]
and define differentials $\sheaf D^i \to \sheaf D^{i+1}$ coming from the inclusions of the components of $X^{i + \dim X + 1}$ into those of $X^{i + \dim X}$, with appropriate alternating signs.

\begin{Proposition}\label{prop:KN_dc}
    Let $X$ be a smooth analytic log variety.
    The dualizing complex $\logdc_{\KN{X}}$ is concentrated in cohomological degree $-\dim X$.
    Further, there exists a isomorphism
    \[
        \Psi(\logdc_{\KN X}[-\dim X]) \cong \bigl( \sheaf D^{-\dim X} \to \dotsc \to \sheaf D^0 \bigr)
    \]
    in $\catDbClogMod{X}$ (where $\sheaf D^{-\dim X}$ is in cohomological degree $0$).
\end{Proposition}

\begin{proof}
    We use the approach of \cite[Remark~5.8]{Koppensteiner:graded_top}.

    As noted above, $\KN X$ is a topological manifold with boundary.
    Let $j\colon X^* \to \KN X$ be the inclusion of the interior.
    Then the dualizing complex of $\KN X$ as an ordinary topological space is $j_!\CC_{X^*}[\dim X]$.

    If $\pi \colon (\KN X, \Lambda, \Clog_X) \to (\KN X, 0, \CC_X)$ is the canonical map, we have for any $\lambda \in \Lambda$
    \begin{equation}\label{eq:prop:KN_dc:1}
        \begin{aligned}
            (\logdc_{\KN{X}})_\lambda
            & \cong \pi_{\gr,*} \sheafHom_{\Clog_X}\bigl(\Clog_X\langle \lambda\rangle,\, \pi_{\gr}^!(j_!\CC_{X^*})\bigr)[\dim X] \\
            & \cong \sheafHom_{\CC}\bigl(\Clog_{X,-\lambda},\, j_!\CC_{X^*}\bigr)[\dim X],
        \end{aligned}
    \end{equation}
    where we note that the functors $\pi_{\gr,*}$ and $j_!$ are exact.

    For each $\lambda \in \Lambda$ there exists an open subset $V_\lambda$ such that $\Clog_{X,\lambda} = i_{\lambda,!}\CC_{V_\lambda}$, where $i_\lambda$ is the inclusion of $V_\lambda$ into $\KN X$.
    Let further $j_\lambda$ be the inclusion of $X^*$ into $V_\lambda$.
    Then from \eqref{eq:prop:KN_dc:1} one obtains
    \[
        (\logdc_{\KN{X}})_\lambda \cong
        i_{-\lambda, *} i_{-\lambda}^!j_! \CC[\dim X] \cong
        i_{-\lambda, *} j_{-\lambda,!}\CC_{X^*}[\dim X].
    \]
    It follows from \cite[Theorem~3.1.2]{Ogus:2003:LogarithmicRiemannHilbertCorrespondence} that $i_{\lambda,*}$ is exact.
    Hence $\logdc_{\KN{X}}$ is concentrated in cohomological degree $-\dim X$.

    One notes that for any point $x$ of $X$, if $(\logdc_{\KN{X}})_\lambda$ is locally constant on $\tau_X^{-1}(x)$, then it has to be actually constant on $\tau_X^{-1}(x)$.
    Thus for $\lambda \notin \logMbargp_X$ one necessarily has $\Psi(\logdc_{\KN{X}})_\lambda = j_!\CC_{X^*}$ (i.e.~the global sections vanish).
    
    To give the required identification of $\Psi(\logdc_{\KN{X}})$ we can assume that $X$ is an affine toric variety.
    Let $v$ be the vertex of $X$.
    Let $X^1_1, \dotsc X^1_n$ be the irreducible components of $X^1$ and let $\lambda_1,\dotsc,\lambda_n$ be the corresponding $\CC$-vector space bases of $\Lambda(X)$ (note that this is in general not a $\ZZ$-basis of $\logMbargp_{X,v}$).
    Write $\lambda = \sum \alpha_i \lambda_i$.
    If there exists some $\alpha_i$ witch is not a negative integer then either $\lambda = 0$ (in which case $\Psi(\logdc_{\KN X}) = j_!\CC_{X^*}$) or $X^* \subsetneq V_{-\lambda} \subsetneq \KN X$.
    In the latter case $i_{-\lambda, *} j_{-\lambda,!}\CC_{X^*}$ is not locally constant on $\tau_X^{-1}(v)$.

    It follows that $\Psi(\logdc_{\KN{X}})$ is isomorphic to the kernel of $\Clog_X \to \bigoplus_i \Clog_{X^1_i}$ (where the components $X^1_i$ are endowed with the induced idealized structure).
    One easily checks that the complex $\sheaf D^{-\dim X} \to \dotsc \to \sheaf D^0$ is exact everywhere else.
\end{proof}

As we alluded to earlier, the logarithmic de Rham functor will have image in $\catDbClogModMon{X}$, the full subcategory of $\catDbClogMod{X}$ consisting of objects with cohomology sheaves in $\catClogModMon{X}$.
\begin{Definition}\label{def:tlogdc}
    Assume $X$ is a smooth log variety.
    The \emph{dualizing complex of $\catDbClogModMon{X}$} is
    \[
        \tlogdc_{\KN X} = \Psi(\logdc_{\KN X}).
    \]
\end{Definition}
This makes sense by Proposition~\ref{prop:KN_dc} and we have an isomorphism
\begin{equation}\label{eq:KN_tdc}
    \tlogdc_{\KN X} \cong \bigl( \sheaf D^{-\dim X} \to \dotsc \to \sheaf D^0 \bigr)
\end{equation}
with $\sheaf D^0$ in cohomological degree $0$.

\subsection{\texorpdfstring{$\sO$}{O}- and \texorpdfstring{$\sD$}{D}-modules on \texorpdfstring{$\KN X$}{the Kato-Nakayama space}}\label{sec:Olog}

In order to define the logarithmic de Rham functor, one needs to enhance the pull back of the structure sheaf $\sO_X$ to $\KN X$ with additional sections.
For this, Kato and Nakayama defined the sheaf $\Olog_X$ \cite[Section 3]{KatoNakayama}, which was then further enlarged by Ogus to the sheaf $\tOlog_X$ \cite[Section~3.3]{Ogus}.
The sheaf of rings $\tOlog_X$ is $\Lambda$-graded with the subsheaf $\Olog_X$ sitting in degree $0$.
It is contained in $\catClogModMon{X}$.

We will not recall the exact definitions here.
Intuitively, if the log structure on $X$ is given by a normal crossings divisor $z_1\dotsm z_\ell = 0$, then $\Olog_X$ should be thought of as formally adding the logarithms $\log z_i$, while $\tOlog_X$ formally adds the powers $z_i^\alpha$, $\alpha \in \CC$ in degree $(0,\dotsc,\alpha,\dotsc,0)$.
If $X$ is a smooth log variety then $\res{\tOlog_X}{X^*}$ is canonically identified with $\sO_{X^*}$.

For any $\sO_X$-module $\sheaf F$ we set
\[
    \tau^*\sheaf F = \tOlog_X \otimes_{\tau^{-1}\sO_X} \tau^{-1} \sheaf F.
\]

\begin{Lemma}
    The sheaf $\tOlog_X$ is flat over $\tau^{-1}\sO_X$.
    In particular, $\tau^*$ is exact.
\end{Lemma}

\begin{proof}
    Flatness can be checked on stalks.
    Thus the statement follows from \cite[Proposition~3.3.5]{Ogus:2003:LogarithmicRiemannHilbertCorrespondence}.
\end{proof}

Let $f\colon X \to Y$ be a morphism of smooth idealized analytic log varieties.
Then there is a commutative diagram of ringed graded spaces.
\[
    \begin{tikzcd}
        (\KN{X}, \Lambda_X, \tOlog_X) \arrow[r, "f_{\log}"] \arrow[d, "\tau_X"] & (\KN Y,\Lambda_Y, \tOlog_Y) \arrow[d, "\tau_Y"] \\
        (X,0,\sO_X) \arrow[r, "f"]{Y} & (Y, 0, \sO_Y)
    \end{tikzcd}
\]
If $f$ is strict then this diagram is a cartesian square.
Recall from \cite[Definition~2.9]{Koppensteiner:graded_top} that a morphism $f$ of graded topological spaces induces a pushforward functor $f_{\gr,*}$ on graded sheaves.
We will need the following theorem from \cite{Gray:thesis:FunctorilityOfLogRH}.

\begin{Theorem}[{\cite[Theorem~2.3.4]{Gray:thesis:FunctorilityOfLogRH}}]\label{thm:tau_and_O-pushfworward}
    Let $f\colon X \to Y$ be a proper morphism of log analytic spaces such that the morphism of characteristics $f^{-1}\logMbargp_Y \to \logMbargp_X$ is injective with torsion free cokernel.
    Then for any $\sheaf F \in \catDbOMod{X}$ the natural isomorphism
    \[
        \tau_Y^* f_* \sheaf F \to f_{\log,\gr,*} \tau_X^* \sheaf F
    \]
    induced by $(\tau^*,\tau_{\gr,*})$-adjunction is an isomorphism.
\end{Theorem}

Following Ogus, we set $\tOmegalogi{i}_X = \tau^*(\bigwedge^i \Omega_X)$.
It is a sheaf of $\tOlog_X$-modules and admits a de Rham differential $d\colon\tOmegalogi{i}_X \to \tOmegalogi{i+1}_X$.
The corresponding complex $(\tOmegalogi{\bullet}, d)$ is quasi-isomorphic to $\Clog_X$ \cite[Proposition~3.3.1]{Ogus}.

In analogy to \cite[Section~3.1.1]{KT} and using the formulas given in \cite[Proposition~3.3.1]{Ogus:2003:LogarithmicRiemannHilbertCorrespondence} we define a sheaf of rings $\tDlog_X$ on $\KN X$ with generators $\tOlog_X$ and $\tau^{-1}(\Theta_X)$.
Clearly $\tOlog_X$ has a natural structure of (left) $\tDlog_X$-module.
Thus if $\sheaf M$ is a $\sD_X$-module, $\tau^*\sheaf M$ has an induced $\tDlog_X$-module structure.
Similarly, if $\sheaf N$ is a right $\sD_X$-module, $\tau^*\sheaf N$ is a right $\tDlog_X$-module.
In particular $\tau^*\sD_X$ is canonically isomorphic to $\tDlog_X$ as a $\tDlog_X$-bimodule and we have an isomorphism
\[
    \tau^*\sheaf M \cong \tDlog_X \otimes_{\tau^{-1}\sD_X} \tau^{-1} \sheaf M.
\]

\begin{Lemma}%
    \label{lem:tau*_and_tensor}%
    \label{lem:tau*_and_hom}%
    For any $\sD_X$-modules $\sheaf M$ and $\sheaf N$ in $\catDbqcDMod{X}$ there is a canonical isomorphisms of $\tDlog_X$-modules
    \begin{align*}
        \tau^*\sheaf M \otimes_{\tOlog_X} \tau^*\sheaf N &\cong \tau^*(\sheaf M \otimes_{\sO_X} \sheaf N)
    \intertext{and if $\sheaf M$ is $\sO_X$-coherent a canonical isomorphism}
        \tau^*\sheafHom_{\sO_X}(\sheaf M, \sheaf N) &\cong \sheafHom_{\tOlog_X}(\tau^*\sheaf M, \tau^* \sheaf N),
    \end{align*}
    where the $\tDlog$-structure on the left hand sides is defined as in Proposition~\ref{prop:tensor_module_struct}.
    In particular, $\tau^*(\canon_X^\dual) \cong \tau^*(\canon_X)^\dual$.
\end{Lemma}

\begin{proof}
    For the tensor product one easily checks the equivalence for the underived functors.
    The derived statement follows as $\tau^*$ is exact and preserves flatness.

    To obtain the statement for $\sheafHom$ it remains to check that the map
    \[
        \tau^{-1}\sheafHom_{\sO_X}(\sheaf M, \sheaf N) \to \sheafHom_{\tau^{-1}\sO_X}(\tau^{-1}\sheaf M, \tau^{-1}\sheaf N)
    \]
    is an isomorphism.
    This can be proved locally, so that we can replace $\sheaf M$ by a free resolution consisting of finite rank modules.
    Hence we are reduced to the case $\sheaf M = \sO_X$ where the statement is trivial.
\end{proof}

\begin{Corollary}\label{cor:tau*canon_resolution}\label{cor:tOlog_resolution}
    One has the following locally free resolutions of the left $\tDlog_X$-module $\tOlog_X$ and the right $\tDlog_X$-module $\tau^*\canon_X$:
    \begin{equation}\label{eq:tOlog_resolution}
        0 \to \tDlog_X \otimes_{\tOlog_X} \bigwedge^{\logdim X} \tau^*\Theta_X \to \cdots \to \tDlog_X \otimes_{\tOlog_X} \bigwedge^0 \tau^*\Theta_X \to \tOlog_X \to 0,
    \end{equation}
    \begin{equation}\label{eq:tau*canon_resolution}
        0 \to \tOmegalogi{0}_X \otimes_{\tOlog_X} \tDlog_X \to \cdots \to \tOmegalogi{\logdim X}_X \otimes_{\tOlog_X} \tDlog_X \to \tau^*\canon_X \to 0.
    \end{equation}
\end{Corollary}

\begin{proof}
    As $\tau^*$ is exact this follows from Corollary~\ref{cor:canon_resolution} by Lemma~\ref{lem:tau*_and_tensor}.
\end{proof}

For a $\tDlog$-module $\sheaf M$ we set
\[
    \tDVerdierlog_X \sheaf M =
    \sheafHom_{\tDlog_X}(\sheaf M,\, \tDlog_X \otimes_{\tOlog_X} \tau^*\twist_X) \otimes_{\tOlog_X} \tau^*(\canon_X^\dual).
\]

\begin{Lemma}\label{lem:tau*_and_duality}
    There exists a natural isomorphism of functors on $\catDbcohDMod{X}$
    \[
        \tau^* \circ \DVerdier_X \cong \tDVerdierlog_X \circ \tau^*.
    \]
    In particular, for any $\sheaf M \in \catDbcohDMod{X}$ there exists a canonical isomorphism $\tDVerdierlog \circ \tDVerdierlog (\tau^*\sheaf M) \cong \tau^*\sheaf M$.
\end{Lemma}

\begin{proof}
    By Lemma~\ref{lem:tau*_and_tensor}, it suffices to prove that for any $\sheaf M \in \catDbcohDMod{X}$ there exists a canonical isomorphism of right $\tDlog_X$-modules
    \[
        \tau^*\sheafHom_{\sD_X}(\sheaf M,\, \sD_X \otimes_{\sO_X} \twist_X)
        \cong
        \sheafHom_{\tDlog_X}(\tau^*\sheaf M,\, \tDlog_X \otimes_{\tOlog_X} \tau^*\twist_X).
    \]
    There is a canonical morphism
    \begin{align*}
        \tau^*\sheafHom_{\sD_X}(\sheaf M,\, \sD_X \otimes_{\sO_X} \twist_X) & =
        \tDlog_X \otimes_{\tau^{-1} \sD_X} \tau^{-1} \sheafHom_{\sD_X}(\sheaf M,\, \sD_X \otimes_{\sO_X} \twist_X) \\ & \to 
        \tDlog_X \otimes_{\tau^{-1} \sD_X} \sheafHom_{\tau^{-1}\sD_X}(\tau^{-1} \sheaf M,\, \tau^{-1}(\sD_X \otimes_{\sO_X} \twist_X)) \\ & \cong
        \sheafHom_{\tau^{-1}\sD_X}(\tau^{-1} \sheaf M,\, \tau^*(\sD_X \otimes_{\sO_X} \twist_X)) \\ & \cong
        \sheafHom_{\tDlog_X}(\tau^* \sheaf M,\, \tDlog_X \otimes_{\tOlog_X} \tau^*\twist_X).
    \end{align*}
    It remains to show that this morphism is an isomorphism.
    This question is local and we can assume that $\sheaf M$ has a resolution by free $\sD_X$-modules.
    Thus we are reduced to showing the statement for $\sheaf M = \sD_X$ which follows immediately from Lemma~\ref{lem:tau*_and_tensor}.
\end{proof}

\section{The de Rham functor}

\begin{Definition}
    Let $X$ be a smooth idealized analytic log variety.
    One defines the \emph{(logarithmic) de Rham functor} $\tDR_X\colon \catDbDMod{X} \to \catDbClogModMon{X}$ by
    \[
        \sheaf M \mapsto \tau^*\canon_X \otimes_{\tDlog} \tau^*\sheaf M.
    \]
    If $X$ is a smooth idealized algebraic log variety, one sets $\tDR_X(\sheaf M) = \tDR_{X^\an}(\sheaf M^\an)$.
\end{Definition}

By Corollary~\ref{cor:tau*canon_resolution} the logarithmic de Rham functor is equivalent to 
\[
    \tDR_X(\sheaf M) \cong \tOmegalogi{\bullet}_X \otimes_{\tOlog_X} \tau^*\sheaf M[\logdim X].
\]

\begin{Lemma}\label{lem:DR_and_Ogus}
    With the notation of \cite[Definition~3.4.1]{Ogus:2003:LogarithmicRiemannHilbertCorrespondence}, $\tDR_X\sheaf M$ is isomorphic to $\sheaf V_X(\sheaf M)[\logdim X]$ for any $\sO_X$-coherent $\sD_X$-module $\sheaf M$.
    In other words up to a shift $\tDR_X$ induces Ogus's Riemann--Hilbert correspondence between logarithmic connections and the category $\catLogLocSys{X}$.
\end{Lemma}

\begin{proof}
    This is just the statement of \cite[Theorem~3.4.2~3(b)]{Ogus:2003:LogarithmicRiemannHilbertCorrespondence}.
\end{proof}

\begin{Lemma}\label{lem:DR_as_Hom}
    For $\sheaf M \in \catDbcohDMod{X}$ there exists a canonical isomorphism
    \[
        \tDR_X(\sheaf M) \cong \sheafHom_{\tDlog_X}(\tOlog_X,\, \tau^*\sheaf M)[\logdim X].
    \]
\end{Lemma}

\begin{proof}
    As $\sheafHom_{\tDlog_X}(\tOlog_X,\, \tau^*\sheaf M) \cong \sheafHom_{\tDlog_X}(\tOlog_X,\, \tDlog_X) \otimes_{\tDlog_X} \tau^*\sheaf M$ it suffices to show the existence of the isomorphism for $\sheaf M = \sD_X$.
    The result now follows from a computation using Corollary~\ref{cor:tOlog_resolution}.
\end{proof}

\begin{Proposition}\label{prop:DR_commutes_with_proper_pushforward}
    Assume that $f\colon X \to Y$ is a proper morphism of smooth analytic log spaces such that the induced map on characteristics $f^{-1}\logMbargp_Y \to \logMbargp_X$ is injective with torsion free cokernel.
    Then there exists a canonical isomorphism
    \[
        \tDR_Y \circ f_\spf \cong f_{\log,\gr,*} \circ \tDR_X.
    \]
\end{Proposition}

\begin{proof}
    To start, note that by Theorem~\ref{thm:tau_and_O-pushfworward}
    \begin{align*}
        \tDR_Y(f_\spf \sheaf M) & =
        \tau_Y^*\canon_Y \otimes_{\tDlog_Y} \tau_Y^* f_*( \sD_{Y \from X} \otimes_{\sD_X} \sheaf M) \\&\cong
        \tau_Y^{-1}\canon_Y \otimes_{\tau_Y^{-1} \sD_Y} \tau_Y^* f_*( \sD_{Y \from X} \otimes_{\sD_X} \sheaf M) \\&\cong
        \tau_Y^{-1}\canon_Y \otimes_{\tau_Y^{-1} \sD_Y} f_{\log,\gr,*}\tau_X^*( \sD_{Y \from X} \otimes_{\sD_X} \sheaf M).
    \end{align*}
    As $f$ is proper, so is $f_{\log}$.
    Hence by the projection formula for graded spaces \cite[Proposition~4.13]{Koppensteiner:graded_top} the above is isomorphic to
    \begin{multline*}
        f_{\log,\gr,*} \bigr( f_{\log}^{-1} \tau_Y^{-1}\canon_Y \otimes_{f_{\log}^{-1}\tau_Y^{-1} \sD_Y} \tau_X^*( \sD_{Y \from X} \otimes_{\sD_X} \sheaf M)\bigr) \\ \cong
        f_{\log,\gr,*} \bigr( \tau_X^{-1} f^{-1} \canon_Y \otimes_{\tau_X^{-1} f^{-1} \sD_Y} \tau_X^{-1}\sD_{Y \from X} \otimes_{\tau_X^{-1}\sD_X} \tau_X^*\sheaf M\bigr).
    \end{multline*}
    As in the classical setting, one has a canonical isomorphism $f^{-1} \canon_Y \otimes_{f^{-1}\sD_Y} \sD_{Y \from X} \cong \canon_X$.
    Hence the above isomorphic to
    \[
        f_{\log,\gr,*} \bigr( \tau_X^{-1} \canon_X \otimes_{\tau_X^{-1}\sD_X} \tau_X^*\sheaf M)\bigr) \cong
        f_{\log,\gr,*} \bigr( \tDR_X(\sheaf M)\bigr).
        \qedhere
    \]
\end{proof}

\subsection{Finiteness}

Write $\catDbfClogMod{X}$ for the subcategory of $\catDbClogMod{X}$ consisting of all $\sheaf F$ such that for all $x \in \KN X$ the stalk $\sheaf F_x$ is finitely generated over $\Clog_{X,x}$.

\begin{Theorem}\label{thm:finiteness}
    Let $X$ be a smooth idealized log variety and let $\sheaf M \in \catDbholDMod{X}$.
    Then $\tDR(\sheaf M) \in \catDbfClogMod{X}$.
\end{Theorem}

In the following, $D^n$ is an $n$-dimensional polydisk with the trivial log structure.

\begin{Lemma}\label{lem:straighten_divisor}
    Let $X$ be a smooth idealized analytic log variety.
    Let $Z$ be a divisor and fix a point $x \in Z$.
    Let $X^k$ be the smallest closed log stratum with $x \in X^k$ and assume that $X^k \cap Z$ is contained in a divisor of $X^k$.
    Then there exists a smooth analytic idealized log variety $X'$ and locally around $x$ an isomorphism $X \cong X' \times D^1$ such that the restriction of the projection $X \to X'$ to $Z$ is finite.
\end{Lemma}

\begin{proof}
    Set $d = \dim X$.
    By \cite[Corollary~IV.3.3.4]{logbook}, locally at $x$ we can assume that $X$ is of the form $Y \times D^{d-k}$ for some $k$-dimensional idealized log variety $Y$ such that $X^k = Y^k \times D^{d-k}$.
    Set $\tilde Z = X^k \cap Z$.
    Then, using the Weierstra\ss\ Preparation Theorem, locally around $x$ there exists an identification $\ul X^k \cong \tilde X' \times D^1$ with $\tilde X'$ (classically) smooth such that the restriction of the projection $\ul X^k \to \tilde X'$ to $\tilde Z$ is finite.
    Set $X' = Y \times \tilde X'$ and let $p$ be the projection $X = X' \times D^1 \to X'$.
    Then by upper semicontinuity for a small enough neighborhood of $x$ the restriction of $p$ to $Z$ is finite .
\end{proof}

\begin{Lemma}\label{lem:fg_and_grading_restriction}
    Let $\Lambda$ be an abelian group, $R$ a $\Lambda$-graded ring and $M$ a $\Lambda$-graded $R$-module.
    For a subgroup $\Lambda' \subseteq \Lambda$ write $R'$ for the $\Lambda'$-graded subring of $R$ generated by the homogeneous elements of $R$ with degrees in $\Lambda'$.
    Analogously, let $M'$ be the submodule of $M$ generated by homogeneous elements of $M$ with degree in $\Lambda'$.
    If $M$ is finitely generated as an $R$-module, then $M'$ is finitely generated as an $R'$-module.
\end{Lemma}

\begin{proof}
    Let $\phi\colon R \to R'$ and $\psi\colon M \to M'$ be the projection maps.
    Then for $r \in R$ and $m \in M' \subseteq M$ one has $\psi(rm) = \phi(r)m$.

    Assume for contradiction that $M'$ is not finitely generated.
    Then there exists an infinite sequence $m_1, m_2, \dotsc$ in $M'$ such that the submodules $M_n' = \langle m_1,\dots,\,m_n \rangle$ form an infinite asscending sequence.
    On the other hand, as $M$ is finitely generated, there must exist an index $N$ such that $RM'_N = RM'_{N+1} = \dots$.
    Thus, for any $n > N$ there exists $r_i \in R$ with $m_n = \sum_{i=1}^N r_im_i$.
    But then also $m_n = \psi(m_n) = \sum_{i=1}^N \phi(r_i)m_i$ and the sequence $M_n'$ stabilizes.
\end{proof}

We call a sheaf $\sheaf F \in \catClogMod{X}$ \emph{poorly constructible} if there exists a stratification $\KN X = \bigsqcup_\alpha X_\alpha$ by subanalytic subsets such that all restrictions $\res{\Lambda}{X_\alpha}$, $\res{\Clog}{X_\alpha}$ and $\res{\sheaf F}{X_\alpha}$ are locally constant as ordinary sheaves.
A complex in $\catDbClogMod{X}$ is poorly constructible if all its cohomology sheaves are.

\begin{Lemma}\label{lem:proper_pf_and_fg}
    Let $f\colon X \to Y$ be a proper morphism of log varieties and assume that the induced map $f^{-1}\logMbargp_Y \to \logMbargp_X$ is injective.
    Let $\sheaf F \in \catDbClogMod{X}$ be poorly constructible and assume that for each $x \in \KN X$ the stalk $\sheaf F_x$ is finitely generated over $\Clog_{X,x}$.
    Then for each $y \in \KN Y$ the stalk $(f_{\log,\gr,*}\sheaf F)_y$ is finitely generated over $\Clog_{Y,y}$.
\end{Lemma}

\begin{proof}
    Let $Z = (\KN X,\, f^{-1}\Lambda_Y,\, f^{-1}\Clog_Y)$, and let $g\colon \KN X \to Z$ and $h \colon Z \to \KN Y$ be the induced maps.
    Then the fact that $h$ preserves finite generation of stalks is classical upon base-changing to a point $y \in \KN Y$, while for $g$ it follows from Lemma~\ref{lem:fg_and_grading_restriction}.
\end{proof}

\begin{proof}[Proof of Theorem~\ref{thm:finiteness}]
    It clearly suffices to prove the theorem in the case that $X$ is analytic.
    As $\tau^*\canon_X$ has a finite length projective resolution over $\tDlog_X$, $\tDR_X$ respects the boundedness condition of the derived categories.
    We have to show that all stalks of $\tDR_X(\sheaf M)$ are finitely generated.

    It clearly suffices to show the statement for $\sheaf M$ in the abelian category $\catHolDMod{X}$.
    We induct on the dimension of $X$.
    If $\dim X = 0$, then $\sheaf M$ is necessarily $\sO_X$-coherent and the statement is an immediate consequence of \cite[Theorem~3.4.2]{Ogus:2003:LogarithmicRiemannHilbertCorrespondence}.
    Thus assume that the theorem holds for all smooth idealized analytic log varieties of dimension strictly less that $\dim X$.

    As $\sheaf M$ is holonomic, $\logdim \Ch(\sheaf M) = \logdim X$ and hence generically $\Ch(\sheaf M)$ is contained in the zero section of $T^*X$.
    Thus there exists a divisor $Z$ of $X$ such that $\sheaf M$ is an integrable connection away from $Z$.
    As $\sheaf M$ is holonomic, so are the restrictions to log strata, and hence by \cite[Lemma~3.17]{KT} the restrictions $\res{\Ch(\sheaf M)}{X^k}$ are generically contained in the zero section.
    Thus $Z \cap X^k$ has at least codimension $1$ in $X^k$.

    Let now $x \in X$ be any point.
    If $x \notin Z$, then $\sheaf M$ is an integrable connection in a neighborhood of $x$ and hence for points of $\KN X$ lying over $x$ the statement follows from \cite[Theorem~3.4.2]{Ogus:2003:LogarithmicRiemannHilbertCorrespondence}.
    Otherwise by Lemma~\ref{lem:straighten_divisor}, there exists a smooth analytic idealized log variety $X'$ with $\dim X' = \dim X - 1$ such that locally $X = X' \times D^1$ and the restriction of the projection $X \to X'$ to $Z$ is finite.
    Shrinking $X$ if necessary, we can assume that $\sheaf M$ has a global good filtration.
    
    As $\res{p}{Z}$ is finite, there exists a radius $0< r < 1$ such that $\{x\} \times \{z : |z| = r\}$ does not intersect $Z$.
    Thus by shrinking $X'$ and $D^1$ if necessary we can assume that the closure of $Z$ in $X' \times \overline D^1$ does not intersect $X' \times \partial D^1$.

    It follows that $(X' \times D^1) \setminus Z$ is a retraction of $(X' \times \CC) \setminus Z$ in a way that is compatible with the log structure.
    Hence \cite[Theorem~3.4.2]{Ogus:2003:LogarithmicRiemannHilbertCorrespondence} implies that the integrable connection $\res{\sheaf M}{X'\times D^1\setminus Z}$ can be extended to an integrable connection on $(X' \times \CC)\setminus Z$.
    As this keeps $\sheaf M$ unchanged in a neighborhood of $Z$, this gives an extension of $\sheaf M$ to a holonomic D-module on $X' \times \CC$.
    Shrinking $X'$ if necessary we can extend $\sheaf M$ to a coherent D-module on $X' \times \ps 1$, where we give $\ps 1$ the log structure defined by just the point at infinity.
    Using \cite[Lemma~3.18]{KT}, we can further assume that $\res{\Ch(\sheaf M)}{X' \times \{\infty\}}$ is contained in the zero section.
    
    To summarize, we can assume that we are in the following situation:
    $\sheaf M$ is a holonomic D-module on $X = X' \times \ps 1$ (with the log structure of $\ps 1$ as above) with a global good filtration and there exists a divisor $Z$ of $X$, not intersecting $X' \times \{\infty\}$, such that the restriction of $\sheaf M$ to $U = X \setminus Z$ is an integrable connection and the restriction of the projection $p\colon X \to X'$ to $Z$ is finite.

    Let $\hat U = \tau_X^{-1}(U)$ and $\hat Z = \tau^{-1}(Z)$ be the ringed graded topological spaces obtained by restricting the structure of $(\KN X,\, \Lambda,\, \Clog_X)$ (i.e.~as topological spaces these are $\KN U$, resp.~$\KN Z$, but the latter is potentially endowed with a different sheaf of rings).
    In particular the restriction of the map $p_{\log}\colon \KN{X} \to \KN{X'}$ to $\hat Z$ is finite and strict.
    Write $j\colon \hat U \to \KN{X}$  and $i \colon \hat Z \to \KN{X}$ for the inclusions.
    Set $\sheaf F = \tDR_{X}(\sheaf M)$.

    We obtain a distinguished triangle (see \cite[Lemma~4.12]{Koppensteiner:graded_top})
    \[
        j_{\gr,!}j_{\gr}^{-1} \sheaf F
        \to
        \sheaf F
        \to
        i_{\gr,!}i_{\gr}^{-1} \sheaf F.
    \]
    Applying $p_{\log,\gr,*} = p_{\log,\gr,!}$ this becomes
    \begin{equation}\label{eq:pf:thm:finiteness:triangle}
        (p_{\log} \circ j)_{\gr,!} j_{\gr}^{-1} \sheaf F
        \to
        p_{\log,\gr,*} \sheaf F
        \to
        (p_{\log} \circ i)_{\gr,!}i_{\gr}^{-1} \sheaf F.
    \end{equation}
    Since $\hat U$ is open, we have $j_{\gr}^{-1} \sheaf K \cong \tDR_{U}(\res{\sheaf F}U)$, which has finitely generated stalks by \cite[Theorem~3.4.2]{Ogus:2003:LogarithmicRiemannHilbertCorrespondence}.
    As $\hat Z$ is a subanalytic subset of $\KN X$ it follows that $j_{\gr,!}j_{\gr}^{-1}\sheaf F$ is poorly constructible and has finitely generated stalks.
    Thus by Lemma~\ref{lem:proper_pf_and_fg} $(p_{\log} \circ j)_{\gr,!} j_{\gr}^{-1} \sheaf F$ also has finitely generated stalks.

    Further, since $p$ is a proper projection, Proposition~\ref{prop:DR_commutes_with_proper_pushforward} implies that
    \[
        p_{\log,\gr,*}\sheaf F = p_{\log,\gr,*}\tDR_{X' \times \ps 1}(\sheaf M) = \tDR_{X'}(p_\spf \sheaf M).
    \]
    As $\sheaf M$, and hence also $p_\spf \sheaf M$, is holonomic (Theorem~\ref{thm:projective_pf_preserves_holonomic}), by the induction hypothesis $p_{\log,\gr,*} \sheaf F$ has finitely generated stalks.
    Thus from the triangle \eqref{eq:pf:thm:finiteness:triangle} one sees that so has $(p_{\log} \circ i)_{\gr,!}i_{\gr}^{-1} \sheaf F$.
    As $p_{\log} \circ i$ is strict and finite, one sees immediately by base changing to a point that $i_{\gr}^{-1} \sheaf F = i_{\gr}^{-1}\tDR_X(\sheaf M)$ also has finitely generated stalks.
    So in particular this is true for any points lying over $x$.
\end{proof}

\subsection{Duality}

As in Definition~\ref{def:tlogdc} we let $\tlogdc_{\KN X} = \Psi(\logdc_{\KN X})$ be the dualizing complex of $\catDbClogModMon{X}$.
The corresponding duality functor is
\[
    \tDD_{\KN X} = \sheafHom_{\Clog_X}({-},\tlogdc_{\KN X}).
\]

\begin{Lemma}
    Let $X$ be a smooth idealized analytic log variety.
    There exists a canonical morphism
    \[
        \tDR_X \circ \DVerdier_X \to \tDD_{\KN X} \circ \tDR_X
    \]
    of functors from $\catDbcohDMod{X}$ to $\catDbClogModMon{X}$.
\end{Lemma}

\begin{proof}
    As Lemma~\ref{lem:hom_and_dual} also holds for $\tDlog_X$-modules, for any $\sheaf M \in \catDbcohDMod{X}$ there exists a canonical isomorphism
    \begin{align*}
        \tDR_X(\DVerdier_X\sheaf M) & =
        \tau^*\canon_X \otimes_{\tDlog_X} \tau^*(\DVerdier_X\sheaf M) \\ &\cong
        \tau^*\canon_X \otimes_{\tDlog_X} \tDVerdierlog_X \tau^*\sheaf M \\ & \cong
        \sheafHom_{\tDlog_X}(\tau^*\sheaf M,\, \tau^*\twist_X).
    \end{align*}
    Thus it suffices to construct a morphism $\sheafHom_{\tDlog_X}(\tau^*\sheaf M,\, \tau^*\twist_X) \to \tDD_{\KN X} \tDR_X(\sheaf M)$.

    As noted before, the kernel of the differential $\tOlog_X \to \tOmegalogi{1}_X$ is identified with $\Clog_X$.
    This is by definition the same as the center of $\tDlog_X$, so that the category of $\tDlog_X$-modules is enriched in $\Clog_X$-modules.
    Consider the canonical morphism
    \[
        \sheafHom_{\tDlog_X}(\tOlog_X,\tau^*\sheaf M) \otimes_{\Clog_X} \sheafHom_{\tDlog_X}(\tau^*\sheaf M,\, \tau^*\twist_X) \to \sheafHom_{\tDlog_X}(\tOlog_X,\, \tau^*\twist_X).
    \]
    By Lemma~\ref{lem:DR_as_Hom} this gives a morphism
    \[
        \tDR_X(\sheaf M) \otimes_{\Clog_X} \sheafHom_{\tDlog_X}(\tau^*\sheaf M,\, \tau^*\twist_X) \to \tDR_X(\twist_X).
    \]
    By Lemma~\ref{lem:DR_and_Ogus}, $\tDR_X$ is exact on integrable connections.
    Thus Corollary~\ref{cor:complex_for_twist} and Proposition~\ref{prop:KN_dc} imply that $\tDR_X(\twist_X)$ isomorphic to the dualizing complex $\tlogdc_{\KN X}$ of $\catDbClogModMon{X}$.
    Hence one obtains a morphism
    \[
        \sheafHom_{\tDlog_X}(\tau^*\sheaf M,\, \tau^*\twist_X) \to 
        \sheafHom_{\Clog_X}(\tDR_X(\sheaf M),\, \tlogdc_{\KN X}) =
        \tDD_{\KN X}\tDR_X(\sheaf M).
        \qedhere
    \]
\end{proof}

\begin{Theorem}\label{thm:DR_and_duality}
    Let $X$ be a smooth idealized log variety and let $\sheaf M \in \catDbholDMod{X}$.
    Then the canonical morphism $\tDR_X \circ \DVerdier_X (\sheaf M) \to \tDD_{\KN X} \circ \tDR_X (\sheaf M)$ is an isomorphism.
\end{Theorem}

\begin{proof}
    As in the classical situation, taking duality commutes with analytification. 
    Hence it suffices to prove the theorem for analytic log varieties.
    Then, as in the proof of Theorem~\ref{thm:finiteness}, one uses induction to reduce to the case of an integrable connection, which is treated in the following lemma.
\end{proof}

\begin{Lemma}\label{thm:DR_and_duality_for_connections}
    Let $X$ be a smooth idealized analytic log variety and let $\sheaf M$ be an integrable connection on $X$
    Then the  canonical morphism $\tDR_X \circ \DVerdier_X (\sheaf M) \to \tDD_{\KN X} \circ \tDR_X (\sheaf M)$ is an isomorphism.
\end{Lemma}

\begin{proof}
    Set $\sheaf F = \tDR(\sheaf M)$.
    Then by \cite[Theorem~3.4.2.3a]{Ogus:2003:LogarithmicRiemannHilbertCorrespondence}, there is a natural identification $\tau^*\sheaf M \cong \tOlog_X \otimes_{\Clog_X} \sheaf F$ and hence
    \begin{align*}
        \sheafHom_{\tDlog_X}(\tau^*\sheaf M,\, \tau^*\twist_X) & \cong
        \sheafHom_{\tDlog_X}(\tOlog_X \otimes_{\Clog_X} \sheaf F,\, \tau^*\twist_X) \\&  \cong
        \sheafHom_{\Clog_X}(\sheaf F,\, \sheafHom_{\tDlog_X}(\tOlog_X,\, \tau^*\twist_X)) \cong
        \tDD_{\KN X}\sheaf F.
        \qedhere
    \end{align*}
\end{proof}

\section{V-filtration and grading}

Let $\ul X$ be a classically smooth complex variety and $\ul D$ a smooth divisor on $X$.
Let $\sheaf J$ be the ideal defining $\ul D$.
This data induces a filtration on $\sD_{\ul X}$ given by
\[
    (V_k\sD_{\ul X})_x = \{ \theta \in \sD_{X,x} : \theta \sheaf J_x^i \subseteq \sheaf J_x^{i-k} \text{ for all } i \in \ZZ \},
\]
where one sets $\sheaf J^i = \sO_X$ for $i \le 0$.
If $X = (\ul X, \ul D)$ is the associated log variety, one notes that $V_0\sD_{\ul X} = \sD_X$.
We will always write $t$ for a local equation of $D$.

We fix a total order on $\CC$ compatible with the usual order on the real numbers.
Let $\sheaf M$ be any holonomic $\sD_{\ul X}$-module.
Then $\sheaf M$ has a canonical Kashiwara--Malgrange V-filtration which we normalize by the requirement that $(t\frac{\partial}{\partial t} + \alpha)$ acts nilpotently on $\gr_\alpha \sheaf M = V_\alpha \sheaf M / V_{<\alpha}\sheaf M$.
We refer to the classical texts \cite{MebkhoutSabbah:1989:DModulesEtCyclesEvanescents} and \cite{Sabbah:1987:DModulesEtCyclesEvanescents} for a detailed introduction to the theory of V-filtrations.

\begin{Lemma}\label{lem:V_holonomic}
    If $\sheaf M$ is a holonomic $\sD_{\underline X}$-module, then $V_\alpha\sheaf M$ is a holonomic $\sD_X$-module for any $\alpha \in \CC$.
\end{Lemma}

\begin{proof}
    By definition, each $V_\alpha\sheaf M$ is $\sD_X$-coherent and $\res{V_\alpha\sheaf M}{X\setminus X^1}$ is holonomic.
    
    By \cite[Corollaire~4.6.3]{MebkhoutSabbah:1989:DModulesEtCyclesEvanescents} each $\res{\gr_\alpha \sheaf M}{X^1}$ is holonomic over $\sD_{\ul X^1}$ and hence by Lemma~\ref{lem:single_stratum_hol} also over $\sD_{X^1}$.
    
    Let us first assume that $\alpha < 0$.
    By \cite[Lemma~3.17]{KT}, it suffices to check that each $\sO_{X^1} \otimes_{\sO_X} V_\alpha\sheaf M$ is $\sD_{X^1}$-holonomic.
    Then $t\colon V_\alpha\sheaf M \to V_{\alpha-1}\sheaf M$ is an isomorphism, so that $\sO_{X^1} \otimes_{\sO_X} V_\alpha\sheaf M$ is a finite sum of modules of the form $\gr_\mu \sheaf M$ for $\alpha-1 < \mu \le \alpha$.
    Hence it is holonomic.
    
    For the general case fix some $\alpha' < \min(0,\alpha)$.
    Then $V_\alpha\sheaf M$ is a finite iterated extension of $V_{\alpha'}\sheaf M$ by sheaves of the form $\gr_\mu\sheaf M$ and hence is holonomic.
\end{proof}

The $\Lambda$-grading on $\tDR_X(\sheaf M)$ induces a filtration, which we will also denote by $V_\lambda$.
The main theorem of this section is as follows.

\begin{Theorem}\label{thm:filtrations_compatible}
    Let $\sheaf M$ be a holonomic $\sD_{\ul X}$-module.
    Then $\tDR_X(V_\alpha \sheaf M) = V_\alpha(\tDR_X\sheaf M)$ for any $\alpha \in \CC \cong \Lambda(X)$.
\end{Theorem}

Identify the sections of $\Clog_X$ locally over points of $D$ with $\CC[s]$ with $\deg s = -1$.
Further locally we have $\Olog_X \cong \tau^{-1}\sO_X[\sigma]$ (with $\deg \sigma = 0$) and $\tOlog_X$ is obtained from $\Olog$ by adding $x_\lambda$ in degree $\lambda$ for all $\lambda \in \CC$.
We note that the inclusion $\Clog_X \subseteq \tOlog_X$ identifies $s^n$ with $t^nx_{-n}$.

\begin{Lemma}\label{lem:dt=td}
    Let $\sheaf M$ be any $\sD_X$-module.
    The differentials of the complex $\tDR_X(\sheaf M)$ are $s$-linear.
\end{Lemma}

\begin{proof}
    This is simply a consequence of the fact that $s$ is contained in the center of $\tDlog_X$
\end{proof}

\begin{Lemma}\label{lem:filtration_building_blocks_and_grading}
    Let $\sheaf M$ be a holonomic $\sD_X$-module, set-theoretically supported on $D$ and fix $\alpha \in \CC$.
    Assume that the minimal polynomial of the action of $t\frac{\partial}{\partial t} \in \sD_X$ on $\sheaf M$ is a power of $(X+\alpha)$.
    Then $\tDR_X(\sheaf M)$ is concentrated in $\Lambda$-degree $\alpha$.
\end{Lemma}

\begin{proof}
    The question is local, so we use the descriptions of $\Clog_X$ and $\tOlog_X$ given above.

    Let $\bar\psi$ be a local section of $H^{k-\dim X}(\tDR_X(\sheaf M))$, represented by a form
    \[
        \psi = \sum_{\substack{\mu \in \CC\\ \ell \in \NN}} \omega_{\mu,\ell} \otimes x_\mu \otimes \sigma^{\ell} \otimes m_{\mu,\ell}, \qquad 
        \omega_{\mu,\ell} \in \Omega^k,\, m_{\mu,\ell} \in \sheaf M
    \]
    with $d^k(\psi) = 0$, where $d^\bullet$ is the differential of the complex $\tDR_X(\sheaf M)$.
    We have to show that there exists a representative of $\bar\psi$ which is entirely contained in degree $\alpha$, i.e.~all of whose terms contain $x_{\alpha}$.

    Since the differential $d^\bullet$ of $\tDR_X(\sheaf M)$ respects the grading, we can assume that $\psi$ is contained in a single degree $\mu$.
    Assume that $\mu \ne \alpha$.

    Let now $\phi$ be any $(k-1)$-form of $\Lambda$-degree $\mu$ in $\tDR_X(\sheaf M)$,
    \[
        \phi = \sum_{\ell \in \NN} \omega'_{\ell} \otimes x_\mu \otimes \sigma^{\ell} \otimes m_{\ell}'.
    \]
    Assume that the differential forms $\omega'_{\ell}$ do not contain $\frac{dt}{t}$, i.e.~$(t\frac{\partial}{\partial t},\omega'_\ell)=0$.
    Then,
    \begin{align*}
        d^{k-1}(\phi)  &= 
        \sum_{\ell} \frac{dt}{t} \wedge \omega'_{\ell} \otimes \biggl(
            \mu x_\mu \otimes \sigma^{\ell} \otimes m_{\ell}' +
            x_\mu \otimes \ell \sigma^{\ell - 1} \otimes m_{\ell}' +
            x_\mu \otimes \sigma^{\ell} \otimes t\frac{\partial}{\partial t} \cdot m_{\ell}'
        \biggr)\\
        & =
        \sum_{\ell} \frac{dt}{t} \wedge \omega'_{\ell} \otimes \biggl(
            x_\mu \otimes \sigma^{\ell} \otimes (t\frac{\partial}{\partial t} + \mu) \cdot m_{\ell}' +
            x_\mu \otimes \ell \sigma^{\ell - 1} \otimes m_{\ul\ell}'
        \biggr).\\
    \end{align*}
    By assumption the operator $t\frac{\partial}{\partial t} + \mu$ is invertible, and hence there exists sections $m'_{\ell}$ of $\sheaf M$ such that $(t\frac{\partial}{\partial t} + \mu) \cdot m'_{\ell} = m_{\ell}$.
    Thus we can recursively (on $\ell$) modify $\psi$ such that it does not contain a term containing $\frac{dt}{t}$.

    We then easily compute
    \[
        \bigl(t\frac{\partial}{\partial t},\, d(\psi)\bigr) =
        \sum_{\ell} \frac{dt}{t} \wedge \omega_{\ell} \otimes \biggl(
            x_\mu \otimes \sigma^{\ell} \otimes (t\frac{\partial}{\partial t} + \mu) m_{\ell} +
            x_\mu \otimes \ell \sigma^{\ell - 1} \otimes m_{\ul\ell}'
        \biggr).\\
    \]
    This has to vanish by assumption.
    Let $\ell'$ be the largest appearing  exponent, which exists as $\sheaf M$ is holonomic.
    Then necessarily $(t\frac{\partial}{\partial t} + \mu) m_{\ell'} = 0$.
    But this is only possible if $\mu = \alpha$.
\end{proof}

\begin{Lemma}\label{lem:filtrations_subset}
    Let $\sheaf M$ be a holonomic $\sD_X$-module.
    Then $\tDR_X(V_\alpha \sheaf M) \subseteq V_\alpha(\tDR_X(\sheaf M))$ for every $\alpha \in \CC$.
\end{Lemma}

\begin{proof}
    First consider the case that $\alpha < 0$.
    We assume for contradiction that there exists a $\mu > \alpha$ such that $\tDR_X(V_\alpha\sheaf M)_\mu \ne 0$.
    Pick a maximal such $\mu$.
    This exists because $V_\alpha\sheaf M$ is $\sD_X$-holonomic by Lemma~\ref{lem:V_holonomic} and hence $\tDR_X(V_\alpha\sheaf M)$ is locally finitely generated as a $\Clog_X$-module by Theorem~\ref{thm:finiteness}.

    Since $\alpha < 0$ multiplication by $t$ induces an isomorphism $V_\alpha\sheaf M \to V_{\alpha - 1}\sheaf M$.
    Thus Lemma~\ref{lem:dt=td} implies that multiplication by $s$ induces an isomorphism $\tDR_X(V_\alpha\sheaf M)_{\mu+1} \to \tDR_X(V_{\alpha-1}\sheaf M)_\mu$.
    Hence maximality of $\mu$ implies that $\tDR_X(V_{\alpha-1}\sheaf M)_\mu = 0$.

    The short exact sequence
    \[
        0 \to V_{\alpha-1}\sheaf M \to V_{\alpha}\sheaf M \to \rquot{V_\alpha\sheaf M}{V_{\alpha-1}\sheaf M} \to 0
    \]
    induces a distinguished triangle of $\CC$-modules on $\KN X$
    \[
        \tDR_X(V_{\alpha-1}\sheaf M)_\mu \to \tDR_X(V_{\alpha}\sheaf M)_\mu \to \tDR_X\left(\rquot{V_\alpha\sheaf M}{V_{\alpha-1}\sheaf M}\right)_\mu.
    \]
    We have already established that the first term vanishes.
    We note that the third term is a finite iterated extension of terms of the form $\tDR_X(\gr_\nu \sheaf M)$ for $\alpha-1 < \nu \le \alpha$.
    Thus, by Lemma~\ref{lem:filtration_building_blocks_and_grading}, $\tDR_X\bigl(\rquot{V_\alpha\sheaf M}{V_{\alpha-1}\sheaf M}\bigr)$ has no degree $\mu$ component and hence the third term vanishes too.
    Therefore also $\tDR_X(V_{\alpha}\sheaf M)_\mu = 0$ in contradiction to the assumption.

    Now if $\alpha$ is general index, we can pick $\alpha' < \min(\alpha,0)$.
    Then $V_\alpha\sheaf M$ is a finite iterated extension of $V_{\alpha'}\sheaf M$ by sheaves of the form $\gr_\mu\sheaf M$ for $\alpha' < \mu \le \alpha$.
    Hence the statement follows form the previous case and Lemma~\ref{lem:filtration_building_blocks_and_grading}.
\end{proof}

\begin{proof}[Proof of Theorem~\ref{thm:filtrations_compatible}]
    We already know that $\tDR_X(V_\alpha \sheaf M) \subseteq V_\alpha(\tDR_X(\sheaf M))$.
    Suppose there exists some $\alpha \in \CC$ for which this inequality is strict.
    Then the difference $\sheaf G$ is in $\Lambda$-degrees at most $\alpha$.
    But by Lemma~\ref{lem:filtration_building_blocks_and_grading} for any $\alpha' > \alpha$, $\tDR_X(V_{\alpha'} \sheaf M)$ agrees with $\tDR_X(V_{\alpha} \sheaf M)$ in $\Lambda$-degrees at most $\alpha$.
    Thus $\sheaf G \subseteq  V_{\alpha'}(\tDR_X(\sheaf M)) \setminus \tDR_X(V_{\alpha'} \sheaf M)$ for all $\alpha' > \alpha$.

    On the other hand, as $\tDR_X$ commutes with colimits, both sides are exhaustive filtrations of $\tDR_X(\sheaf M)$, contradicting what was just observed.
\end{proof}

\sloppy
\printbibliography

\end{document}